\documentclass{siamart190516}
\usepackage[T1]{fontenc}
\usepackage{graphicx} % Required for inserting images
\usepackage{amsfonts,verbatim,cite,subcaption,tikz,float,multicol,url,booktabs,mathabx}
\usepackage[section]{placeins}
\usepackage[algo2e,ruled,linesnumbered]{algorithm2e}

\newcommand{\Ttran}{\mathsf{T}}

\newcommand{\Htran}{\mathsf{H}}

\newcommand{\de}{\,\mathrm{d}}
\newcommand{\Tr}{\mathsf{Tr}}

\newcommand{\prob}{\mathsf{P}}
\newcommand{\expt}{\mathsf{E}}

\newcommand{\mi}{\!\cdot\!\mathsf{i}}
\newcommand{\fro}{\mathsf{F}}
\newcommand{\defi}{:=}
\newcommand{\van}{\mathsf{Van}}

\newcommand{\han}{\mathsf{Han}}
\newcommand{\normml}{\left\vert\kern-0.25ex\left\vert\kern-0.25ex\left\vert}  
\newcommand{\normmr}{\right\vert\kern-0.25ex\right\vert\kern-0.25ex\right\vert}
\newcommand{\order}{\mathcal{O}}
\newcommand{\rat}{\mathcal{R}}

\newcommand{\R}{\mathbb{R}}
\newcommand{\C}{\mathbb{C}}
\newcommand{\F}{\mathbb{F}}
\newcommand{\rk}{\mathrm{rank}}

\newcommand{\gap}{\mathsf{gap}}

\newcommand{\mpr}{\mathbf{u}}

\newcommand{\fig}{eps}

\newcommand{\poly}{\mathcal{P}}
\newcommand{\region}{\mathcal{D}}

\newcommand{\figsizeD}{0.45\textwidth}

\providecommand{\spa}[1]{\mathsf{span}\{#1\}}

\providecommand{\diagm}{\mathrm{diag}}
\providecommand{\diagM}[1]{\mathrm{diag}(#1)}
\providecommand{\abs}[1]{\lvert#1\rvert}
\providecommand{\norm}[1]{\lVert#1\rVert}

\providecommand{\Bigabs}[1]{\Bigl\lvert#1\Bigr\rvert}

\providecommand{\bignorm}[1]{\bigl\lVert#1\bigr\rVert}
\providecommand{\Bignorm}[1]{\Bigl\lVert#1\Bigr\rVert}

\newtheorem{remark}[theorem]{Remark}
\crefname{algocf}{Algorithm}{Algorithms}

\numberwithin{equation}{section}

\begin{document}
\title{Convergence analysis of a nonlinear eigensolver based on rational approximation of the resolvent
    }
    \author{
        Nian Shao\thanks{Institute of Mathematics, EPF Lausanne, 1015 Lausanne, Switzerland (\href{mailto:nian.shao@epfl.ch}{nian.shao@epfl.ch})}
        \and
        Yuji Nakatsukasa\thanks{Mathematical Institute, University of Oxford, Oxford, OX2 6GG, UK (\href{nakatsukasa@maths.ox.ac.uk}{nakatsukasa@maths.ox.ac.uk}).}
        }
    \headers{Eigensolver based on rational approximation of resolvent}{Nian~Shao and Yuji Nakatsukasa}
\maketitle
\begin{abstract}
Given a holomorphic matrix-valued function, the poles of its sketched resolvent are generically its eigenvalues. 
Once a good rational approximation of the sketched resolvent is obtained, the poles of this rational approximation typically lie close to those eigenvalues, thus providing a flexible framework for solving both linear and nonlinear eigenvalue problems.
However, the accuracy of the computed eigenvalues is limited and remains poorly understood. 
This paper analyzes the convergence of this approach and demonstrates the effectiveness of two techniques to improve accuracy: block probing and zooming in. 
We also establish the backward and forward stability of polefinding for a barycentric rational form via a generalized eigenproblem.
Numerical experiments demonstrate the sharpness of our theoretical results.
\end{abstract}

\section{Introduction}
This work is concerned with nonlinear eigenvalue problems (NEPs) \cite{Mehrmann2004,Guttel2017} of the form
\begin{equation}\label{eq:mainproblem}
    T(\lambda) v = 0,\quad \lambda\in\region\quad\text{and}\quad v\in\C^{n}\setminus\{0\}.
\end{equation}
Throughout this paper, we assume that \(\region\subset\C\) is a nonempty domain enclosed by a Jordan curve, and $T$ is an $n\times n$ holomorphic matrix-valued function on some domain that contains the closure of $\region$. 
We also assume that the NEP~\cref{eq:mainproblem} is regular, that is, that  $\det T(z_{0})\neq 0$ for some $z_{0}\in\region$.
In this paper, we consider an approach that computes eigenvalues through the poles of rational approximations of the resolvent $T^{-1}$.
To our knowledge, such an approach was first proposed in \cite{Austin2015} for symmetric eigenvalue problems, and later developed in \cite{bruno2024evaluation} for the evaluation of resonances. A general discussion can be found in~\cite{nakatsukasa2025}. 

Let $\Omega_{L}$ and $\Omega_{R}$ be two $n\times b$ probing matrices, such as Gaussian random matrices. 
Our approach starts with the following surrogate function for the resolvent:
\begin{equation*}
  F(z) \defi \Omega_{L}^{\Htran}T(z)^{-1}\Omega_{R},\quad 
  z\in\region.
\end{equation*} 
For generic $\Omega_{L}$ and $\Omega_{R}$, the poles of the surrogate function $F$ are exact eigenvalues of $T$.
Let $R$ be a good rational approximation of $F$ in $\region$. Then we expect that the eigenvalues of $T$ are approximated by the poles of $R$. Once we obtain an approximate eigenvalue $\lambda$, the approximate eigenvector can be computed by SVD of $T(\lambda)$ or residual inverse iteration \cite{Neumaier1985}.

The main ingredient of our approach is computing a rational approximation. The AAA (adaptive Antoulas-Anderson) algorithm \cite{Nakatsukasa2018} provides an efficient tool for rational approximation, and has been employed in a number of applications \cite{nakatsukasa2025}.
The AAA algorithm basically computes a barycentric representation of the rational approximation by a greedy strategy. For a matrix-valued function, the set-valued AAA method \cite{Lietaert2022} computes a matrix-valued rational approximation in a similar approach.
Unlike existing contour integral-based approaches \cite{Beyn2012,Polizzi2009,Gavin2018,Sakurai2003,kressner2026linear}, which interpret the quadrature rule as a rational approximation of the indicator function of $\region$, our method directly approximates the resolvent itself and extracts the eigenvalues from its poles.
Moreover, our method differs from the nonlinear eigensolvers based on rational approximation proposed in \cite{GNT2022,Guttel2024,SEM2019,Guttel2014}, which approximate surrogate functions of $T$ instead of its resolvent $T^{-1}$.

Even though the (set-valued) AAA can compute a near-best rational approximation effectively, the accuracy of the (nonlinear) eigensolver is not always satisfactory.
Consider a toy example where $T(z)=\diagm\{0.1,\dotsc,0.9\}-zI_{9}$ and $\Omega_{L}$ and $\Omega_{R}$ are two complex Gaussian random vectors. 
We uniformly sample 50 points on the circle $\{z\in\C \mid \abs{z-1/2}=1/2\}$ and apply AAA to compute a rational approximation of the sketched resolvent, achieving an error of about $10^{-14}$ at the sample points. The real parts\footnote{The nonzero imaginary
parts result from the fact that AAA does not maintain real symmetry.} of the poles computed by AAA are listed below, where the incorrect digits are underlined:
\begin{equation*}
\begin{aligned}
  \widehat{\lambda}_{1} &= 0.10000000000000\underline{26}, \quad
  \widehat{\lambda}_{2} = 0.2000000000000\underline{690}, \quad
  \widehat{\lambda}_{3} = 0.2999999999\underline{865939}, \\
  \widehat{\lambda}_{4} &= 0.39999999999\underline{63088}, \quad
  \widehat{\lambda}_{5} = 0.5000000000\underline{123708}, \quad
  \widehat{\lambda}_{6} = 0.59999999999\underline{71500}, \\
  \widehat{\lambda}_{7} &= 0.69999999999\underline{76688}, \quad
  \widehat{\lambda}_{8} = 0.8000000000000\underline{127}, \quad
  \widehat{\lambda}_{9} = 0.90000000000000\underline{32}.
\end{aligned}
\end{equation*}
Repeating the procedure with different sketching vectors or increasing the number of samples yields similar results, indicating that even for well-separated eigenvalues, the naive approach is unable to compute them with high accuracy, especially for those interior eigenvalues.
Besides that, another unclear aspect of this polefinding approach is how to detect multiple eigenvalues. Even in symmetric eigenvalue problems, the original algorithms in \cite{bruno2024evaluation, Austin2015}, which employ probing vectors rather than probing matrices, cannot reveal the multiplicity of an eigenvalue. This limitation arises because the scalar-valued resolvent possesses only simple, first-order poles, even when the underlying eigenvalue has higher multiplicity.

All these limitations suggest that this approach requires certain modifications for reliable computation of eigenvalues.
In this paper, we focus on the following two effective strategies: zooming in and block probing.
The zooming-in technique, proposed in \cite{bruno2024evaluation}, addresses situations where too many eigenvalues are contained within the domain $\region$, and empirically also improves the accuracy. The idea is to recursively apply the polefinding procedure to dyadic subdivisions of $\region$, thereby reducing the number of eigenvalues in each subdomain.
The block probing strategy, which is already widely used in other algorithms for eigenvalue problems \cite{Guttel2024,Ikegami2010,kressner2024randomized}, replaces the probing vectors used in \cite{bruno2024evaluation,Austin2015} with probing matrices. This modification is expected to enable the detection of eigenvalue multiplicities and improve numerical robustness.
By leveraging the techniques of zooming in and block probing, the polefinding approach becomes significantly more stable, leading to the algorithm summarized in \cref{algo}. We remark that this method is essentially replacing the probing vectors in \cite[Algo.~2]{bruno2024evaluation} with probing matrices.
Applying this modified method to the same problem as above, with block size $2$ and $50$ sampling points taken on the circles $\{z\in\C\mid \abs{z-c}=1/6\}$ for $c=1/6$, $c=1/2$ and $c=5/6$, we obtain the real parts of the computed eigenvalues shown below, which exhibit at least $15$ digits of accuracy.
\begin{equation*}
\begin{aligned}
  \widehat{\lambda}_{1} &= 0.1000000000000000, \quad
  \widehat{\lambda}_{2} =  0.1999999999999999, \quad
  \widehat{\lambda}_{3} = 0.300000000000000\underline{3}, \\
  \widehat{\lambda}_{4} &= 0.400000000000000\underline{2}, \quad
  \widehat{\lambda}_{5} = 0.500000000000000\underline{2}, \quad
  \widehat{\lambda}_{6} = 0.600000000000000\underline{2}, \\
  \widehat{\lambda}_{7} &= 0.700000000000000\underline{2}, \quad
  \widehat{\lambda}_{8} = 0.7999999999999999, \quad
  \widehat{\lambda}_{9} = 0.9000000000000004.
\end{aligned}
\end{equation*}

\begin{algorithm2e}[htbp]
    \caption{Nonlinear eigensolver based on rational approximation of resolvent}
    \label{algo}
    \SetKwProg{Fn}{Function}{:}{}
    \KwIn{Matrix-valued function $T(\lambda)$, domain $\region$, block size $b$ of probing\;} 
    Draw two $n\times b$ Gaussian random matrices $\Omega_{L}$ and $\Omega_{R}$ in $\F$\;
    Construct a function handle $F(z) = \Omega_{L}^{\Htran}T(z)^{-1}\Omega_{R}$\tcp*[r]{Block probing}

    \Return{$\Lambda= \texttt{AdaPoleFinding} (F,\region)$}\;

    \Fn{$\texttt{AdaPoleFinding} (F,\region)$}{
      $\Lambda= \texttt{PoleFinding} (F,\region)$\;
      Partition $\region$ into small domains $\region_{i}$ and set $\Lambda_{i}=\Lambda\cap \region_{i}$\;
      \For{each unconverged $\Lambda_{i}$}{
        $\Lambda_{i}= \texttt{AdaPoleFinding} (F,\region_{i})$\tcp*[r]{Zooming in} 
      }
    }
    \Return{$\Lambda=\cup_{i}\Lambda_{i}$}\;

    \Fn{$\texttt{PoleFinding} (F,\region)$    }
    {
      Sample a sequence of points $\{z_{k}\}$ in $\region$\;
      Apply (set-valued) AAA to $\{F(z_{k}),z_{k}\}$ to find a rational approximation $R$ in a barycentric form\;
      Find roots $\Lambda$ of the denominator of $R$\;
    }
    \Return{$\Lambda$}\;
\end{algorithm2e}

From a theoretical perspective, we address three questions related to \cref{algo}:
\begin{enumerate}
  \item Can we recover all spectral information of $T$ from the surrogate function $F$?\label{Q1}
  \item Will the poles of the rational approximation $R$ in line~13 converge to those of the surrogate function $F$, given that $R$ provides a good approximation only at the sample points or $\partial\region$? \label{Q2}
  \item Is the pole/root finding procedure in line~14 stable? \label{Q3}
\end{enumerate}
In this paper, we answer all three questions in the affirmative, by analyzing  the effects of block probing and zooming in.
To the best of our knowledge, no prior analysis is available on the accuracy improvements that either of these techniques can bring to resolvent-based eigensolvers. Specifically, we show that block probing enables the recovery of spectral information even when $T$ possesses multiple eigenvalues, and that the combination of block probing and zooming in can significantly enhance the accuracy of the computed eigenvalues.

The rest of this paper is organized as follows.
We begin with Question~\ref{Q1} by studying the effect of sketching in \Cref{sec:sketching}.
When the block size $b$ is at least as large as the geometric multiplicity of an eigenvalue $\lambda$ of $T$, we show that the multiplicity information of $\lambda$ can be inferred from the asymptotic behavior of the singular values of $F(z)$ as $z$ approaches $\lambda$.
Moreover, \cref{thm:prob} shows that block probing can significantly reduce the probability of missing eigenvalues from sketching.

\Cref{sec:convergence} addresses Question~\ref{Q2} about the convergence of poles.
We first establish a general convergence result in \cref{thm:convPolesIn}, which shows that a naive approach may lead to inaccurate results.
We then justify the effectiveness of the zooming-in technique. On the one hand, we show that the accuracy of the poles improves linearly as the diameter of the computational domain decreases. On the other hand, this technique also reduces the number of poles in the domain. The main theoretical result of this section, \cref{thm:singlepole}, explains the substantial improvement in accuracy resulting from reducing the number of poles, particularly when combined with block probing.

Finally, Question~\ref{Q3} is discussed in \Cref{sec:stability}, where the backward and forward stability of finding the poles of $r$ via a generalized eigenvalue problem
under mild conditions 
are established in \cref{thm:bs} and \cref{thm:fs}, respectively.
Numerical experiments are also presented to illustrate the theoretical results, along with their sharpness in most cases, and the corresponding source code are publicly available at \url{https://github.com/nShao678/AAA-EVP-code}.

We remark that this work focuses on the theoretical analysis of \cref{algo}, and in particular 
does not provide an optimized implementation, which itself warrants independent investigation.
One challenge in this direction is the design of a reliable, and possibly tailored, \texttt{cleanup} procedure for removing spurious poles, also known as Froissart doublets, from the (set-valued) AAA algorithm.

\paragraph{Notation}
For a scalar function $f$ defined on $\region$, we denote its $L^{\infty}$-norm on $\region$ by $\norm{f}_{\region}$.
We denote its $k$th order derivative by $f^{(k)}$.
The boundary of $\region$ is denoted by $\partial \region$. The diameter of $\region$ is the maximal distance between any two points in $\region$.
For a matrix input, $\norm{\cdot}$ denotes the spectral norm, equal to the largest singular value. 
We say that $\Omega$ is a real Gaussian matrix if all its entries are independent and identically distributed (i.i.d.) according to the normal distribution $\mathcal{N}(0,1)$. Similarly, $\Omega$ is a complex Gaussian matrix if its real and imaginary parts take i.i.d. entries according to the distribution $\mathcal{N}(0,1/2)$.
The machine precision (or unit roundoff) is denoted by $\mpr$. The notation $\cdot^{\Ttran}$ and $\cdot^{\Htran}$ denote the transpose and Hermitian transpose, respectively. We denote the field by $\F$, where $\F=\R$ for real and $\F=\C$ for complex.
We use $a = \order(b)$ and $a = \Theta(b)$ to indicate $\abs{a} \leq c \abs{b}$ and $c_{1} \abs{b} \leq \abs{a} \leq c_{2} \abs{b}$ for some positive constants $c$, $c_{1}$, and $c_{2}$, respectively.
\section{Sketching of resolvent for nonlinear eigenvalue problems}
\label{sec:sketching}
We start with the classical Keldysh theorem \cite{keldysh1971completeness}, which provides a characterization of the resolvent.
\begin{theorem}[Keldysh]
    \label{thm:Keldysh}
    Let $\lambda_{1},\dotsc,\lambda_{s}$ be the distinct eigenvalues of $T$ in $\region$, and define 
    \begin{equation*}
        m\defi \sum_{i=1}^{s}m_{i},\quad \text{where}\quad m_{i}\defi\sum_{j=1}^{d_{i}}m_{ij},
    \end{equation*}
    where $m_{i,1}\geq m_{i,2}\geq \dotsb \geq m_{i,d_{i}}$  denote the  partial (algebraic) multiplicities of $\lambda_{i}$ for $i=1,\dotsc,s$.
    Then there are $n\times m$ matrices $V_{L}$ and $V_{R}$, and an $m\times m$ Jordan matrix $J$ with eigenvalues $\lambda_{i}$ of partial multiplicities $m_{ij}$, such that 
    \begin{equation*}
        T(z)^{-1}=V_{L}(zI-J)^{-1}V_{R}^{\Htran}+H(z)
    \end{equation*}
    for all $z\in\region\setminus\{\lambda_{1},\dotsc,\lambda_{s}\}$,
    where $H$ is a holomorphic function in $\region$, and 
    \begin{small}
    \begin{equation*}
        J=\begin{bmatrix}
            J_{1}&&\\ 
            &\ddots&\\ 
            &&J_{s}
        \end{bmatrix},
        \quad 
        J_{i}=\begin{bmatrix}
            J_{i,1}&&\\ 
            &\ddots&\\ 
            &&J_{i,d_{i}}
        \end{bmatrix},
        \quad J_{ij} = \begin{bmatrix}
            \lambda_{i} &1&&\\ 
            &\ddots&\ddots&\\
            &&\lambda_{i}& 1\\
            &&&\lambda_{i}
        \end{bmatrix}\in\C^{m_{ij}\times m_{ij}}.
    \end{equation*} 
    \end{small}
    Moreover, let 
\begin{equation*}
    V_{L}=[V_{L,1},\dotsc,V_{L,s}]\quad\text{and}\quad  V_{R}=[V_{R,1},\dotsc,V_{R,s}],\quad \text{where}\quad V_{L,i},V_{R,i}\in\C^{n\times m_{i}}.
\end{equation*}
Then $\rk(V_{L,i}),\rk(V_{R,i})\geq d_{i}$.
\end{theorem}

\begin{remark}
    In \cref{thm:Keldysh}, $m_{i}$ and $d_{i}$ are called the algebraic and geometric multiplicity of $\lambda_{i}$, respectively.
    The matrices $V_{L}$ and $V_{R}$ could be wide matrices, that is, $m>n$.
    The columns of matrices $V_{L,i}$ and $V_{R,i}$ are called generalized eigenvectors, which consist of Jordan chains of each eigenvalue. We refer to \cite[Sec.~2]{Guttel2017} for details.
\end{remark}

\subsection{Preserving spectral information by block probing}

Let $\Omega_{L}$ and $\Omega_{R}\in\F^{n\times b}$ be two probing matrices, such as Gaussian random matrices. As mentioned above, the surrogate function for the resolvent is constructed as 
\begin{equation}
    \label{eq:defF}
    F(z) \defi \Omega_{L}^{\Htran}T(z)^{-1}\Omega_{R} = \Omega_{L}^{\Htran}V_{L}(z I-J)^{-1}V_{R}^{\Htran}\Omega_{R}+H_{\Omega}(z),
\end{equation}
where $H_{\Omega}(z)=\Omega_{L}^{\Htran}H(z)\Omega_{R}$ is holomorphic in $\region$.
When $m_{ij}=1$ for all $1\leq i\leq s$ and $1\leq j\leq d_{i}$, that is, $J$ is diagonal, then all eigenvalues $\lambda$ of $T$ in $\mathcal{D}$ are non-defective, and would be equal to the poles of $F$ for generic probing matrices, that is, all $n\times b$ matrices except on a set of measure zero.
The following theorem generalizes this result to the defective setting, which is related to, but different from \cite[Thm.~3.3]{Beyn2012}.
Here we provide the asymptotic growth of the singular values when $z$ is close to certain eigenvalues of $T$, whereas \cite[Thm.~3.3]{Beyn2012} focuses primarily on the structure of the canonical system of generalized eigenvectors.
\begin{theorem}
    \label{thm:sval}
    Consider the matrix-valued function $T$ in \cref{thm:Keldysh} and the surrogate function $F$ in \cref{eq:defF}.
    Assume that the block size of the probing matrices $\Omega_{L}$ and $\Omega_{R}$ satisfies $b\geq d_{i}$, where $d_{i}$ is the geometric multiplicity of $\lambda_{i}$.
    Denote the largest $d_{i}$ singular values of $F(\lambda_{i}+\epsilon)$ by $\sigma_{j}(\epsilon)$ for $1\leq j\leq d_{i}$.
    For generic probing matrices, when $\epsilon\to 0$, it holds that $\sigma_{j}(\epsilon)^{-1}=\Theta(\epsilon^{m_{ij}})$.
\end{theorem}
\begin{proof}
    Note that 
    \begin{equation*}
            F(z) = \sum_{i=1}^{s}\Omega_{L}^{\Htran}V_{L,i}(zI-J_{i})^{-1}V_{R,i}^{\Htran}\Omega_{R}+H_{\Omega}(z).
    \end{equation*}
    Since $H_{\Omega}$ is holomorphic and the geometric multiplicity of $\lambda_{i}$ is $d_{i}$, we know that $F(z)$ has at most $d_{i}$ singular values blowing up to $\infty$ as $z\to\lambda_{i}$. 
    For generic probing matrices in $\F^{n\times b}$ with $b\geq d_{i}$, both $\Omega_{L}^{\Htran}V_{L,i}$ and $\Omega_{R}^{\Htran}V_{R,i}\in\F^{b\times m_{i}}$ are two generic matrices of rank at least $d_{i}$ because both $V_{L,i}$ and $V_{R,i}\in\F^{n\times m_{i}}$ of rank at least $d_{i}$.
    Therefore, such generic probing matrices do not cancel the blow-up as $z\to\lambda_{i}$, and in turn, $F(z)$ has exactly $d_{i}$ singular values blowing up as $z\to\lambda_{i}$.

    To find its asymptotic rate, it is sufficient to look into the singular values of $\bigl((\lambda_{i}+\epsilon)I_{m_{i}}-J_{i}\bigr)^{-1}$, which consist of singular values of $\bigl((\lambda_{i}+\epsilon)I_{m_{ij}}-J_{ij}\bigr)^{-1}$ for $1\leq j\leq d_{i}$.
    Since $\lambda_{i} I_{m_{ij}}-J_{ij}$ has a one-dimensional null space, only the largest singular value of its inverse blows up. 
    Noting that 
    \begin{equation*}
        \bigl((\lambda_{i}+\epsilon) I_{m_{ij}}-J_{ij}\bigr)^{-1} = \epsilon^{-1}\bigl(I_{m_{ij}}-(J_{ij}-\lambda_{i}I_{m_{ij}})/\epsilon\bigr)^{-1} = \sum_{k=1}^{m_{ij}}\frac{1}{\epsilon^{k}}(J_{ij}-\lambda_{i} I_{m_{ij}})^{k-1},
    \end{equation*}
    we know that $\bignorm{\bigl((\lambda_{i}+\epsilon)I_{m_{ij}}-J_{ij}\bigr)^{-1}}^{-1}=\Theta(\epsilon^{m_{ij}})$. Then the theorem is proved by the monotonicity of $m_{ij}$.
\end{proof}

\cref{thm:sval} stands as the theoretical foundation of \cref{algo}, showing that the surrogate function $F$ keeps all essential information of the spectrum of $T$ if the block size $b$ is at least the same as the maximum geometric multiplicities. One can find all eigenvalues of $T$ by finding the poles of $F$ and detecting the partial multiplicities from the asymptotic behavior of the growth of its singular values. As an example, consider the following matrix 
\begin{equation}\label{eq:2example}
    J = \mathsf{blkdiag}\Biggl\{\begin{bmatrix}
        2&1&\\ 
        &2&1\\ 
        &&2
    \end{bmatrix},\begin{bmatrix}
        2&1&\\ 
        &2&1\\ 
        &&2
    \end{bmatrix},\begin{bmatrix}
        2&1\\ 
        &2
    \end{bmatrix},2,2\Biggr\}\in\R^{10\times 10}
\end{equation}
containing an eigenvalue $2$ with geometric multiplicity $5$, and partial multiplicities  $m_{1,1}=m_{1,2} = 3$, $m_{1,3}=2$, $m_{1,4}=m_{1,5}=1$.
Sketching with complex Gaussian random matrices of size $10\times 6$, we collect numerical results in \cref{fig:multi}. We see indeed that the first five singular values of $\Omega_{L}^{\Htran}(J-(2+\epsilon)I)^{-1}\Omega_{R}$ blow up like $\epsilon^{-3}$, $\epsilon^{-3}$, $\epsilon^{-2}$, $\epsilon^{-1}$ and $\epsilon^{-1}$, respectively, reflecting  the theoretical results in \cref{thm:sval}.
\begin{figure}[htbp]
    \centering
    \includegraphics[width=0.5\textwidth]{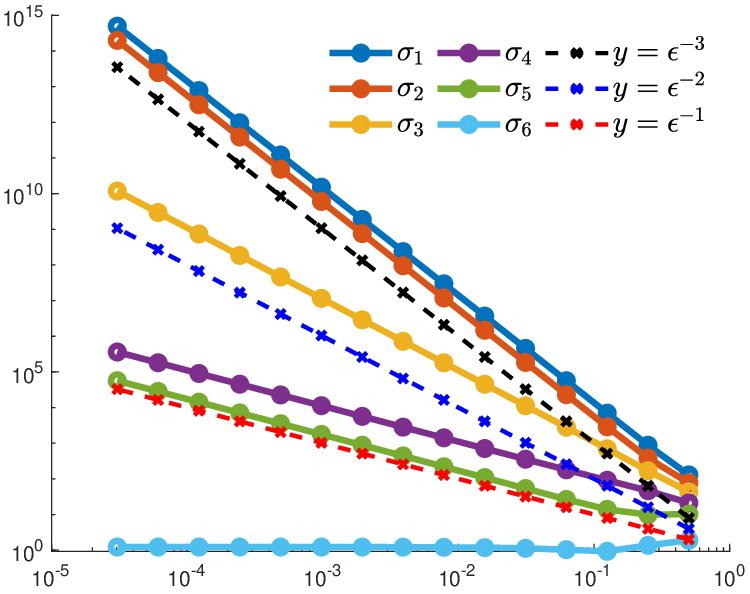}
    \caption{Asymptotic growth of singular values for $\Omega_{L}^{\Htran}(J-(2+\epsilon)I)^{-1}\Omega_{R}$, for the $J$ in~\cref{eq:2example}. 
    The $x$-axis is $\epsilon$.}
    \label{fig:multi}    
\end{figure}

\subsection{Robustness from block probing}

Besides including information about the multiplicity and defectivity of eigenvalues, using probing matrices with large block size can also improve the robustness. We first further partition  $V_{L,i}$ and $V_{R,i}$ in \cref{thm:Keldysh} as  
\begin{equation*}
    V_{L,i} = [V_{L,i,1},\dotsc,V_{L,i,d_{i}}]\quad\text{and}\quad 
    V_{R,i} = [V_{R,i,1},\dotsc,V_{R,i,d_{i}}],
\end{equation*}
where $V_{L,ij},V_{R,ij}\in\C^{n\times m_{ij}}$.
Then by \cref{thm:Keldysh}, we have 
\begin{equation*}
    T(z)^{-1} = \sum_{i=1}^{s}\sum_{j=1}^{d_{i}}V_{L,ij}(zI-J_{ij})^{-1}V_{R,ij}^{\Htran}+H(z) = \sum_{i=1}^{s}\sum_{j=1}^{d_{i}}\sum_{k=1}^{m_{ij}}S_{ijk}(z-
    \lambda_{i})^{-k}+H(z),
\end{equation*}
where, letting $V_{L,ij}=[v_{L,ij}^{0},v_{L,ij}^{1},\dotsc,v_{L,ij}^{m_{ij}-1}]$ and $V_{R,ij}=[v_{R,ij}^{m_{ij}-1},\dotsc,v_{R,ij}^{1},v_{R,ij}^{0}]\in\C^{n\times m_{ij}}$,
\begin{equation*}
    S_{ijk}=\sum_{\ell=0}^{m_{ij}-k}v_{L,ij}^{\ell}\bigl(v_{R,ij}^{m_{ij}-k-\ell}\bigr)^{\Htran}\in\C^{n\times n}.
\end{equation*}
After sketching in \cref{eq:defF}, we obtain 
\begin{equation}
    \label{eq:defS}
    \begin{aligned}
        F(z) &= \Omega_{L}^{\Htran}T(z)^{-1}\Omega_{R} = \sum_{i=1}^{s}\sum_{j=1}^{d_{i}}\sum_{k=1}^{m_{ij}}\Omega_{L}^{\Htran}S_{ijk}\Omega_{R}(z-\lambda_{i})^{-k}+H_{\Omega}(z)\\ 
        &= \sum_{i=1}^{s}\sum_{k=1}^{m_{i,1}}\Omega_{L}^{\Htran}S_{ik}\Omega_{R}(z-\lambda_{i})^{-k}+H_{\Omega}(z),\quad\text{where}\quad S_{ik}=\sum_{\substack{1\leq j\leq d_{i}\\ k\leq m_{ij}}}S_{ijk}.
    \end{aligned}
\end{equation}

In order to obtain a robust surrogate function, 
we seek to ensure that no eigenvalue is missed due to sketching. A sufficient condition is that every sketched coefficient $\Omega_{L}^{\Htran}S_{ik}\Omega_{R}$ deviates from zero. To show this can be done robustly by block probing, we first prove the following lemma.

\begin{lemma}
    \label{lem:prob}
    Let $S$ be a nonzero matrix in $\F$, and $\omega_{L}$ and $\omega_{R}$ be two i.i.d. Gaussian random vectors in $\F$. For sufficiently small $\epsilon>0$, it holds that 
    \begin{equation*}
        \begin{aligned}
            \prob(\abs{\omega_{L}^{\Htran}S\omega_{R}}&\leq \epsilon\norm{S}) =\order( \epsilon^{\ell}\abs{\log\epsilon}),        
        \end{aligned}
    \end{equation*} 
    where $\ell=1$ if $\F=\R$ and $\ell=2$ if $\F=\C$. Here the hidden constant in $\order(\cdot)$ is universal.
\end{lemma}

\begin{proof}
    Without loss of generality, we assume that $\norm{S}=1$.
    Let us first consider the case where $S$ is a scalar. Using the density function of the magnitude of the product of two independent Gaussian random variables, as given in \cite{Epstein1948} for $\F=\R$ and \cite{Wells1962} for $\F=\C$, and the asymptotic behavior of he modified Bessel function of the second kind $K_{0}$ from \cite{NIST:DLMF},  we have 
    \begin{equation}
        \label{eq:probscalar}
        \begin{aligned}
            \prob(\abs{\omega_{L}^{\Htran}S\omega_{R}}&\leq \epsilon) = \prob\bigl(\abs{\omega_{L}^{\Htran}\omega_{R}}\leq \epsilon\bigr) = \frac{2}{\pi}\int_{0}^{\epsilon} K_{0}(x)\de x 
            = \order(\epsilon\abs{\log\epsilon})\quad \text{for}\quad \F=\R,\\ 
            \prob(\abs{\omega_{L}^{\Htran}S\omega_{R}}&\leq \epsilon) = \prob\bigl(\abs{\omega_{L}^{\Htran}\omega_{R}}\leq \epsilon\bigr) = \int_{0}^{2\epsilon} xK_{0}(x)\de x = \order(\epsilon^{2}\abs{\log\epsilon}) \quad \text{for}\quad \F=\C. 
        \end{aligned}
    \end{equation}
    When $S$ is not a scalar, without loss of generality, we assume that $S$ is a diagonal matrix whose first entry is $1$. Conditional on $\omega_{R}$, $\omega_{L}^{\Htran}S\omega_{R}$ is a Gaussian random variable with mean zero and variance $\norm{S\omega_{R}}^{2}$. 
    When $\F=\C$, the magnitude follows the Rayleigh distribution, yielding 
    \begin{equation*}
        \begin{aligned}
            \prob(\abs{\omega_{L}^{\Htran}S\omega_{R}}\leq \epsilon\mid \omega_{R}) 
            &= \int_{0}^{\epsilon}\frac{2x}{\norm{S\omega_{R}}^{2}}\exp\Bigl(\frac{-x^{2}}{\norm{S\omega_{R}}^{2}}\Bigr)\de x\\ 
            &= 1-\exp\Bigl(\frac{-\epsilon^{2}}{\norm{S\omega_{R}}^{2}}\Bigr) \leq 1-\exp\Bigl(-\frac{\epsilon^{2}}{\abs{\widetilde{\omega}_{R}}^{2}}\Bigr)
        \end{aligned}
    \end{equation*}
    where $\widetilde{\omega}_{R}$ is the first entry of $\omega_{R}$. By the law of total probability, we have 
    \begin{equation*}
        \begin{aligned}
            \prob(\abs{\omega_{L}^{\Htran}S\omega_{R}}\leq \epsilon) 
        &=\expt_{\omega_{R}}\prob(\abs{\omega_{L}^{\Htran}S\omega_{R}}\leq \epsilon\mid \omega_{R}) \\ 
        &\leq 1-\expt_{\omega_{R}} \exp\Bigl(-\frac{\epsilon^{2}}{\abs{\widetilde{\omega}_{R}}^{2}}\Bigr)=1-\expt_{\widetilde{\omega}_{R}} \exp\Bigl(-\frac{\epsilon^{2}}{\abs{\widetilde{\omega}_{R}}^{2}}\Bigr).
        \end{aligned}
    \end{equation*}
    Denote the first entry of $\omega_{L}$ by $\widetilde{\omega}_{L}$. With the same argument as before, we know that 
    \begin{equation*}
        \prob (\abs{\widetilde{\omega}_{L}\widetilde{\omega}_{R}}\leq \epsilon ) 
        =   \expt_{\widetilde{\omega}_{R}} \prob (\abs{\widetilde{\omega}_{L}\widetilde{\omega}_{R}}\leq \epsilon \mid \widetilde{\omega}_{R}) 
        = 1-\expt_{\widetilde{\omega}_{R}}\exp\Bigl(-\frac{\epsilon^{2}   }{\abs{\widetilde{\omega}_{R}}^{2}}\Bigr).
    \end{equation*}
    Combining the two relationships above, we obtain 
    \begin{equation*}
        \prob(\abs{\omega_{L}^{\Htran}S\omega_{R}}\leq \epsilon) \leq 1-\expt_{\widetilde{\omega}_{R}} \exp\Bigl(-\frac{\epsilon^{2}   }{\abs{\widetilde{\omega}_{R}}^{2}}\Bigr)
        = \prob (\abs{\widetilde{\omega}_{L}\widetilde{\omega}_{R}}\leq \epsilon ).
    \end{equation*}
    The proof for $\F=\C$ is completed by using the bound for the scalar case in \cref{eq:probscalar}.

    When $\F=\R$, consider the function $\psi(\alpha) = \alpha\exp(-\alpha^{2} x^{2}/2)$, where $0\leq x^{2}\leq \epsilon^{2}$. Note that 
    \begin{equation*}
        \psi^{\prime}(\alpha) = (1-\alpha^{2} x^{2})\exp(-\alpha^{2} x^{2}/2)\geq 0 \quad\text{if}\quad \alpha\leq \epsilon^{-1}.
    \end{equation*}
    Since $\norm{S\omega_{R}}\geq \abs{\widetilde{\omega}_{R}}$, when $\abs{\widetilde{\omega}_{R}}\geq \epsilon$, we have 
    \begin{equation*}
        \begin{aligned}
            \prob(\abs{\omega_{L}^{\Htran}S\omega_{R}}\leq \epsilon\mid \omega_{R}) 
        &= \int_{0}^{\epsilon}\frac{2}{\sqrt{2\pi}\norm{S\omega_{R}}}\exp\Bigl(\frac{-x^{2}}{2\norm{S\omega_{R}}^{2}}\Bigr)\de x 
        \\ 
        &\leq \int_{0}^{\epsilon}\frac{2}{\sqrt{2\pi}\abs{\widetilde{\omega}_{R}}}\exp\Bigl(\frac{-x^{2}}{2\abs{\widetilde{\omega}_{R}}^{2}}\Bigr)\de x
        = \prob(\abs{\widetilde{\omega}_{L}\widetilde{\omega}_{R}}\leq \epsilon\mid \widetilde{\omega}_{R})
        \end{aligned}
    \end{equation*}
    Using the law of total probability again, we obtain 
    \begin{equation*}
        \begin{aligned}
        &\prob(\abs{\omega_{L}^{\Htran}S\omega_{R}}\leq \epsilon) =  
        \expt_{\omega_{R}}\prob(\abs{\omega_{L}^{\Htran}S\omega_{R}}\leq \epsilon\mid \omega_{R})\\  
        \leq\ &\expt_{\omega_{R}}(\boldsymbol{1}_{\abs{\widetilde{\omega}_{R}}\geq \epsilon}+\boldsymbol{1}_{\abs{\widetilde{\omega}_{R}}< \epsilon})\prob(\abs{\omega_{L}^{\Htran}S\omega_{R}}\leq \epsilon\mid \omega_{R})\\
        \leq\ &\expt_{\omega_{R}}\boldsymbol{1}_{\abs{\widetilde{\omega}_{R}}\geq \epsilon}\prob(\abs{\widetilde{\omega}_{L}\widetilde{\omega}_{R}}\leq \epsilon\mid \widetilde{\omega}_{R}) + \prob(\abs{\widetilde{\omega}_{R}}\leq \epsilon)\\
        \leq\ &\expt_{\widetilde{\omega}_{R}}\prob(\abs{\widetilde{\omega}_{L}\widetilde{\omega}_{R}}\leq \epsilon\mid \widetilde{\omega}_{R}) + \prob(\abs{\widetilde{\omega}_{R}}\leq \epsilon)
        =\prob (\abs{\widetilde{\omega}_{L}\widetilde{\omega}_{R}}\leq \epsilon ) + \prob(\abs{\widetilde{\omega}_{R}}\leq \epsilon).
        \end{aligned} 
    \end{equation*}
    We complete the proof for $\F=\R$ by using the bound for the scalar case in \cref{eq:probscalar} again.
\end{proof}
\begin{remark}
    The $\abs{\log\epsilon}$ factor can be removed by using the second largest singular value of $S$ if it is positive.
    The key observation is that, when $\omega_{L}, \omega_{R} \in \mathbb{F}^{d}$ with $d\geq2$, the small-ball probability for $\abs{\omega_{L}^{\Htran} \omega_{R}} \leq \epsilon$ scales as $\order(\epsilon^{\ell})$, rather than the $\order(\epsilon^{\ell} \abs{\log\epsilon})$ behavior that arises in the scalar case \cref{eq:probscalar}.
    This situation is, however, atypical, since for a non-defective eigenvalue $\lambda_{i}$, the matrix $S_{ik}$ in \cref{eq:defS} has rank one. 
\end{remark}
\begin{remark}
    The real case in \cref{lem:prob} is similar to \cite[Thm.~2.2]{Bujanovic2021}, which provides an anti-concentration bound for $\prob(\abs{\omega_{L}^{\Htran}S\omega_{R}}\leq \epsilon\norm{S}_{\fro})$. However, our result is slightly stronger in the sense that our bound is independent of the size of $S$.
\end{remark}

\cref{lem:prob} studies the robustness when using probing vectors. If one uses probing matrices with block size $b$, the probability of $\norm{\Omega_{L}^{\Htran}S_{ik}\Omega_{R}}_{\fro}$ being close to zero will decrease dramatically. This is because the sum  of several non-negative independent random variables is very unlikely to be close to zero. The following theorem formalizes this argument.

\begin{theorem}
    \label{thm:prob}
    Consider the surrogate function $F$ defined in \cref{eq:defF}, where $\Omega_{L}$ and $\Omega_{R}$ are two i.i.d. $n\times b$ Gaussian random matrices in $\F$. For a sufficiently small $\epsilon>0$ and coefficient matrices $S_{ik}\in\F^{n\times n}$ in \cref{eq:defS}, we have
    \begin{equation*}
        \prob\Biggl( \min_{
        \substack{1\leq i\leq s\\  1\leq k\leq m_{i,1}} }\frac{\norm{\Omega_{L}^{\Htran}S_{ik}\Omega_{R}}_{\fro}}{\norm{S_{ik}}}\leq \epsilon \Biggr)
        = \order\bigl((\epsilon^{\ell}\abs{\log\epsilon})^{b}\bigr),
    \end{equation*} 
    where $\ell=1$ for $\F=\R$ and $\ell=2$ for $\F=\C$. 
\end{theorem}

\cref{thm:prob} shows that $\norm{\Omega_{L}^{\Htran}S_{ik}\Omega_{R}}_{\fro}$ is very unlikely to be significantly smaller than $\norm{S_{ik}}$, especially for slightly larger $b$. This implies that sketching will restore all modes $(z-\lambda_{i})^{-k}$ in \cref{eq:defS} with high probability. Moreover, we remark that $\norm{\Omega_{L}^{\Htran}S\Omega_{R}}_{\fro}$ is closely related to a trace estimator with Kronecker structure. This is because 
\begin{equation*}
    \begin{aligned}
        \norm{\Omega_{L}^{\Htran}S\Omega_{R}}_{\fro}^{2} 
    &= \bignorm{\mathsf{vec}(\Omega_{L}^{\Htran}S\Omega_{R})}^{2}
    = \Tr\Bigl((\overline{\Omega}_{R}\otimes\Omega_{L})^{\Htran} \mathsf{vec}(S)\mathsf{vec}(S)^{\Htran}(\overline{\Omega}_{R}\otimes\Omega_{L})\Bigr)\\ 
    &= \sum_{i,j=1}^{b} \omega_{ij}^{\Htran}(\mathsf{vec}(S)\mathsf{vec}(S)^{\Htran})\omega_{ij},\quad \text{where}\quad  \omega_{ij} = \overline{\omega}_{R,i}\otimes \omega_{L,j}
    \end{aligned}
\end{equation*}
with $\Omega_{L}=[\omega_{L,1},\dotsc,\omega_{L,b}]$ and  $\Omega_{R}=[\omega_{R,1},\dotsc,\omega_{R,b}]$. Compared with existing results in \cite{meyer2023hutchinson,meyer2025understanding}, we focus on anti-concentration results instead of trace estimation.

\begin{proof}[Proof of \cref{thm:prob}]
    By taking the union bound, we know that 
    \begin{equation}
        \label{eq:pfthmprob}
        \begin{aligned}
            &\prob\Biggl( \min_{
        \substack{1\leq i\leq s\\  1\leq k\leq m_{i,1}}}\frac{\norm{\Omega_{L}^{\Htran}S_{ik}\Omega_{R}}_{\fro}}{\norm{S_{ik}}}\leq \epsilon \Biggr)
        \leq 
        \sum_{i=1}^{s}\sum_{k=1}^{m_{ij}}
        \prob\Bigl( \frac{\norm{\Omega_{L}^{\Htran}S_{ik}\Omega_{R}}_{\fro}}{\norm{S_{ik}}}\leq \epsilon \Bigr)\\ 
        \leq\ & \max_{
        \substack{1\leq i\leq s\\  1\leq k\leq m_{i,1}} }
        \prob\Bigl(\frac{\norm{\Omega_{L}^{\Htran}S_{ik}\Omega_{R}}_{\fro}}{\norm{S_{ik}}}\leq \epsilon \Bigr)\sum_{i=1}^{s}m_{ij}
        =m\max_{
        \substack{1\leq i\leq s\\  1\leq k\leq m_{i,1}} }
        \prob\Bigl(\frac{\norm{\Omega_{L}^{\Htran}S_{ik}\Omega_{R}}_{\fro}}{\norm{S_{ik}}}\leq \epsilon \Bigr).
        \end{aligned} 
    \end{equation}
    Let $\Omega_{L}=[\omega_{L,1},\dotsc,\omega_{L,b}]
    $ and $\Omega_{R}=[\omega_{R,1},\dotsc,\omega_{R,b}]$. 
    For any matrix $S\in\F^{n\times n}$ with $\norm{S}=1$, we have 
    \begin{equation*}
        \norm{\Omega_{L}^{\Htran}S\Omega_{R}}_{\fro}^{2} = \sum_{i,j=1}^{b}\abs{\omega_{L,i}^{\Htran}S\omega_{R,j}}^{2}\geq \sum_{i=1}^{b}\abs{\omega_{L,i}^{\Htran}S\omega_{R,i}}^{2}.  
    \end{equation*}
    Note that $\omega_{L,i}^{\Htran}S\omega_{R,i}$ are i.i.d. random variables for $1\leq i\leq b$. 
    
    Using the independence and \cref{lem:prob}, we have 
    \begin{equation}
        \label{eq:probS}
        \begin{aligned}
            \prob \bigl(\norm{\Omega_{L}^{\Htran}S\Omega_{R}}_{\fro}\leq \epsilon \bigr) 
            &\leq  
        \prob \bigl(\sum_{i=1}^{b}\abs{\omega_{L,i}^{\Htran}S\omega_{R,i}}^{2}\leq \epsilon^{2} \bigr)\leq 
        \prod_{i=1}^{b}\prob \bigl(\abs{\omega_{L,i}^{\Htran}S\omega_{R,i}}\leq \epsilon \bigr)\\ 
        & \leq \order\bigl((\epsilon^{\ell}\abs{\log\epsilon})^{b}\bigr).
        \end{aligned}
    \end{equation}
    The proof is completed by plugging it into \cref{eq:pfthmprob}.
\end{proof}

The probability bound in \cref{eq:probS} is sharp up to a factor of $\abs{\log\epsilon}^{b-1}$ when $S$ has rank one, as occurs, for example, for the residue associated with a simple eigenvalue. In the less typical case in which every relevant coefficient matrix $S_{ik}$ has at least two nonnegligible singular values, numerical evidence suggests that the tail probability scales like $\epsilon^{(2b-1)\ell}$; see, for example, \cref{fig:gaussian}. We also remark that \cref{thm:prob} does not cover the case in which $T$ has complex eigenvectors while the sketching matrices are real Gaussian, which is a highly uncommon setup. For example, eigenvalues of standard real nonsymmetric problems are generically complex, so one usually use complex samples. Thus, complex Gaussian matrices are the natural choice for sketching.

\begin{figure}[htbp]
    \centering
    \includegraphics[width=\figsizeD]{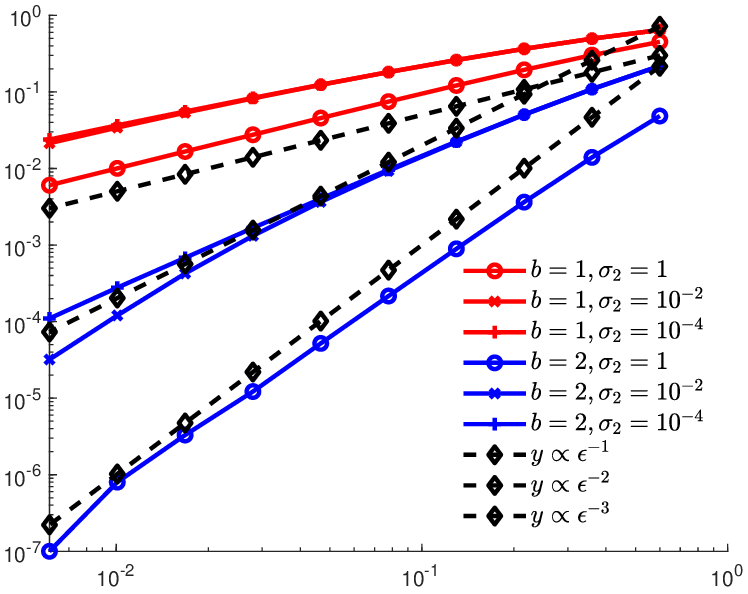}
    \includegraphics[width=\figsizeD]{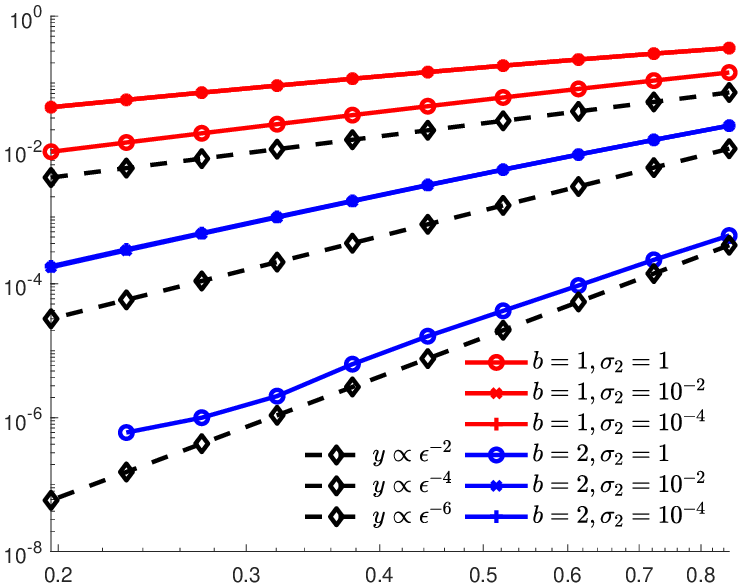}
    \caption{Small ball probability ($10^{7}$ samples) for $\prob(\norm{\Omega_{L}^{\Htran}S\Omega_{R}}_{\fro}\leq \epsilon)$, where $S\in\F^{100\times 100}$ is a rank-two matrix with $\norm{S}=1$ and $\Omega_{L}$ and $\Omega_{R}\in\F^{n\times b}$ for $\F=\R$ (left) and $\F=\C$ (right). 
    }
    \label{fig:gaussian}    
\end{figure}

Although \cref{thm:prob} only provides tail bounds for the Frobenius norm, we can also derive analogous bounds for the singular values.
In particular, if $S_{ik}$ has rank $(m_{ij}-k+1)$ and the block sizes of $\Omega_{L}$ and $\Omega_{R}$ are larger than $(m_{ij}-k+1)$, then the $(m_{ij}-k+1)$-th largest singular value of $\Omega_{L}^{\Htran} S_{ik} \Omega_{R}$ deviates from zero with high probability.
To demonstrate this, consider a rank-$m$ matrix $S \in \F^{n \times n}$ with $m \leq b \ll n$.
Without loss of generality, we assume that $S$ is diagonal with its first $m$ diagonal entries being nonzero.
Let $\Omega_{L}^{(m)}$ and $\Omega_{R}^{(m)}$ denote the first $m$ rows of $\Omega_{L}$ and $\Omega_{R}$, respectively.
Then we have
\begin{equation*}
    \kappa_{m}(\Omega_{L}^{\Htran}S\Omega_{R})\leq \kappa_{m}(S) \kappa_{m}(\Omega_{L}^{(m)})\kappa_{m}(\Omega_{R}^{(m)}),
\end{equation*}  
where $\kappa_{m}(\cdot)$ denotes the ratio between the largest and the $m$-th largest singular values.
Using the tail bounds for the singular values of Gaussian random matrices (see, for example, \cite{Rudelson2010}), we find that $\kappa_{m}(\Omega_{L}^{\Htran} S \Omega_{R})$ is 
bounded by a modest constant multiple of $\kappa_{m}(S)$ with high probability, implying that all $m$ nonzero singular values of $\Omega_{L}^{\Htran} S \Omega_{R}$ are comparable.
Consequently, we can obtain a corresponding bound for the singular values by applying \cref{thm:prob}.\section{Convergence analysis}
\label{sec:convergence}
Given a surrogate function $F$, this section discusses the convergence of poles of its rational approximation $R$. In the following, we assume that the rational (matrix-valued) approximation $R$ is of the following barycentric form
\begin{equation}
    \label{eq:defR}
    R(z) = \sum_{k=1}^{K}\frac{w_{k}F_{k}}{z-z_{k}}\Big/\sum_{k=1}^{K}\frac{w_{k}}{z-z_{k}}\in\C^{b\times b}, 
\end{equation}
where $z_{1},\dotsc,z_{K}\in\region$ are support points, $w_{1},\dotsc,w_{K}\in\C$ are (normalized) weights with $\sum_{k=1}^{K}\abs{w_{k}}^{2}=1$ and $F_{k}=F(z_{k})\in\C^{b\times b}$. 
Such an approximation is computed by the (set-valued) AAA method.
\begin{remark}
    As mentioned in \cite{Gosea2021}, there are two kinds of matrix-valued barycentric forms, where the other one takes matrix-valued coefficients $W_{k}$. We use the form in \cref{eq:defR} because with the scalar-valued denominator, we can restrict our discussion to polefinding for scalar functions.
\end{remark}

Given a meromorphic function $f$ on a bounded domain $\region\subset \C$, it is impossible to expect a sequence of rational functions $r_{n}$ to converges uniformly to $f$ unless they share exactly the same poles in $\region$. 
However, if we assume that $f$ is holomorphic on $\partial\region$, the convergence on $\partial\region$ can be shown to be the so-called $m_{1}$-almost uniform convergence \cite{Goncar1975,Blatt2011}, that is, $r_{n}$ converges uniformly to $f$ in $\region$ except on an arbitrarily small subset covering all poles of $f$. This $m_{1}$-almost uniform convergence further implies that the poles of $r_{n}$ converge to those of $f$. 
In this section, we first quantify the convergence of poles on a unit domain, and then discuss the effect of zooming in through a scaling argument.
Furthermore, we explain the rapid convergence observed when only a few poles lie within the domain $\region$, which is primarily due to the zooming-in technique, especially when block probing is employed.

Throughout this section, we let $\poly_{n}$ denote the space of polynomials of degree at most $n$, and let $\rat_{n,\widetilde{m}_{n}}$ denote the space of rational functions with relatively prime numerator and denominator, having at most $n$ zeros and at most $\widetilde{m}_{n}$ poles in $\region$, counted with multiplicity.
We remark that the notation $n$ is reused here and should not be confused with its previous meaning as the size of $T$.

\subsection{Convergence and effect of zooming in}
In \cite[Lem.~1]{Walsh1964}, Walsh asserted that if a sequence of rational functions $r_{n}$ converges to a meromorphic function $f$ on a deleted neighborhood of its pole $\lambda_{i}$ with multiplicity $m_{i,1}$, then for sufficiently large $n$, $r_{n}$ contains at least $m_{i,1}$ poles in a small neighborhood of $\lambda_{i}$. The following result quantifies this phenomenon.  

\begin{theorem}
    \label{thm:convPolesIn}
    Let $\region\subset\C$ be a bounded connected domain with unit diameter. Assume that $f$ is meromorphic in $\region$ and holomorphic in a neighborhood of $\partial\region$, with poles $\lambda_{1},\dotsc,\lambda_{s}$ of multiplicities $m_{1,1},\dotsc,m_{s,1}$, respectively. Given an $\epsilon>0$, let
    \begin{equation*}
        \Lambda_{i}(\epsilon)\defi\{z\in\region\mid\abs{z-\lambda_{i}}\leq\epsilon^{1/m_{i,1}}\}.
    \end{equation*}
    Let $\widetilde{m}=\sum_{i=1}^{s}m_{i,1}$ and let $r_{n}\in\rat_{n,\widetilde{m}_{n}}$ be a sequence of rational functions satisfying
    \begin{equation*}
        \norm{f-r_{n}}_{\partial\region}\leq\gamma_{n},\qquad\lim_{n\to\infty}\gamma_{n}=0,\qquad\widetilde{m}_{n}\leq\widetilde{m}.
    \end{equation*}
    Then there exist $\epsilon_{0}>0$ and $C>0$, independent of $n$ and $\epsilon$, such that, for every $0<\epsilon\leq\epsilon_{0}$ and every $n$ satisfying $\gamma_{n}\leq C^{-1}\epsilon$, the function $r_{n}$ contains exactly $m_{i,1}$ poles in $\Lambda_{i}(\epsilon)$.
\end{theorem}

In numerical computation, one cannot expect to obtain an arbitrarily small $\gamma_{n}$. Instead, \cref{thm:convPolesIn} should be understood as follows: Suppose we can compute a rational approximation $r_{n}$ of $f$ on $\partial \region$ with accuracy $\gamma_{n}=\order(\mpr)$, for example by AAA, then the approximation accuracy of a pole $\lambda_{i}$ of multiplicity $m_{i,1}$ inside $\region$ is expected to be $(C\mpr)^{1/m_{i,1}}$. 
Another remark is that the multiplicity of a pole corresponds to the largest partial algebraic multiplicity of the eigenvalue, instead of its total algebraic/geometric multiplicity in our primary context, where $f$ is a sketched resolvent of an NEP.
\begin{proof}[Proof of \cref{thm:convPolesIn}]
    We assume that $\epsilon_{0}$ is small enough such that $\Lambda_{i}(\epsilon_{0})$ are pairwise disjoint and compactly contained in $\region$. Let
    \begin{equation*}
        q_{\widetilde{m}_{n}}(z)=\prod_{\substack{1\leq j\leq \widetilde{m}_{n}\\ \mu_{j}\neq\infty}}(z-\mu_{j}),
    \end{equation*}
    where the finite $\mu_{j}$ are precisely the poles of $r_{n}$ in $\region$, counted with multiplicity, and the remaining $\mu_{j}$ are set to $\infty$. Then $q_{\widetilde{m}_{n}}$ is monic, $\deg(q_{\widetilde{m}_{n}})\leq\widetilde{m}_{n}$, and all its zeros lie in $\region$. Moreover, $p_{n}\defi r_{n}q_{\widetilde{m}_{n}}$ is holomorphic in $\region$.
    
    Let
    \begin{equation*}
        q_{f}(z)=\prod_{i=1}^{s}(z-\lambda_{i})^{m_{i,1}}\quad\text{and}\quad  p_{f}(z)=f(z)q_{f}(z).
    \end{equation*}
    Then $p_{f}$ is holomorphic in $\region$ and $p_{f}(\lambda_{i})\neq0$ for $1\leq i\leq s$. Since $\region$ has unit diameter and all zeros of $q_{\widetilde{m}_{n}}$ lie in $\region$, we have $\norm{q_{\widetilde{m}_{n}}}_{\partial\region}\leq 1$.

The first step is to establish the bound \cref{eq:denominatorClose} for $\norm{q_{\widetilde{m}_{n}}-q_{f}}_{\partial\region}$.
    Define
    \begin{equation*}
        E_{n}=p_{f}q_{\widetilde{m}_{n}}-p_{n}q_{f}=(f-r_{n})q_{f}q_{\widetilde{m}_{n}}.
    \end{equation*}
    The function $E_{n}$ is holomorphic in $\region$, and
    \begin{equation}
        \label{eq:rouche1}
        \norm{E_{n}}_{\partial\region}\leq\norm{f-r_{n}}_{\partial\region}\norm{q_{f}}_{\partial\region}\norm{q_{\widetilde{m}_{n}}}_{\partial\region}\leq C_{1}\gamma_{n}
    \end{equation}
    for some $C_{1}>0$ independent of $n$. By the maximum modulus principle and Cauchy's estimates, we have 
    \begin{equation}
        \label{eq:hermiteSmall}
        \abs{E_{n}^{(k)}(\lambda_{i})}\leq C_{2}\gamma_{n}\quad\text{for}\quad  0\leq k<m_{i,1},
    \end{equation}
    for some $C_{2}>0$ independent of $n$. Since $q_{f}$ has a zero of multiplicity $m_{i,1}$ at $\lambda_{i}$, we know that $(p_{n}q_{f})^{(k)}(\lambda_{i})=0$ for $0\leq k<m_{i,1}$.
    It follows from \eqref{eq:hermiteSmall} that
    \begin{equation}
        \label{eq:hermiteSmall2}
        \abs{(p_{f}q_{\widetilde{m}_{n}})^{(k)}(\lambda_{i})}\leq C_{2}\gamma_{n},\quad\text{for}\quad  0\leq k<m_{i,1}.
    \end{equation}

    Consider the linear map
    \begin{equation*}
        \Psi:\poly_{\widetilde{m}}\to\C^{\widetilde{m}},\qquad \Psi(h)=\left((hp_{f})^{(k)}(\lambda_{i})\right)_{\substack{1\leq i\leq s\\0\leq k<m_{i,1}}}.
    \end{equation*}
    If $\Psi(h)=0$, then $hp_{f}$ has a zero of multiplicity at least $m_{i,1}$ at every $\lambda_{i}$. Since $p_{f}(\lambda_{i})\neq0$, the polynomial $h$ has a zero of multiplicity at least $m_{i,1}$ at every $\lambda_{i}$, and hence $q_{f}$ divides $h$. Since $h\in\poly_{\widetilde{m}}$ and $\deg(q_{f})=\widetilde{m}$, it follows that $\mathsf{null}(\Psi)=\spa{q_{f}}$.
    Then using \cref{eq:hermiteSmall2}, we know that there exists a $c_{n}\in\C$, such that
    \begin{equation*}
        \norm{q_{\widetilde{m}_{n}}-c_{n}q_{f}}_{\partial\region}\leq C_{3}\gamma_{n}   
    \end{equation*}
    for some $C_{3}>0$ independent of $n$. By equivalence of norms on the finite-dimensional space $\poly_{\widetilde{m}}$, the coefficients of $q_{\widetilde{m}_{n}}-c_{n}q_{f}$ are $\order(\gamma_{n})$. Since $q_{\widetilde{m}_{n}}$ and $q_{f}$ are monic, comparison of their leading coefficients shows that, for sufficiently small $\gamma_{n}$, $\deg(q_{\widetilde{m}_{n}})=\widetilde{m}$ and $c_{n}=1+\order(\gamma_{n})$. Consequently,
\begin{equation}
    \label{eq:denominatorClose}
    \norm{q_{\widetilde{m}_{n}}-q_{f}}_{\partial\region}\leq C_{4}\gamma_{n}
\end{equation}
for some $C_{4}>0$ independent of $n$.

    Now we are ready to establish the convergence of poles.
    For $z\in\partial\Lambda_{i}(\epsilon)$, we have
    \begin{equation*}
        \abs{q_{f}(z)}=\epsilon\prod_{\substack{1\leq j\leq s\\j\neq i}}\abs{z-\lambda_{j}}^{m_{j,1}}.
    \end{equation*}
    Since the poles $\lambda_{1},\dotsc,\lambda_{s}$ are distinct, there exists $C_{5}>0$ such that
    \begin{equation}
        \label{eq:qfLower}
        \min_{z\in\partial\Lambda_{i}(\epsilon)}\abs{q_{f}(z)}\geq C_{5}\epsilon,\qquad 0<\epsilon\leq\epsilon_{0}.
    \end{equation}
    By the maximum modulus principle and \eqref{eq:denominatorClose}, we have 
    \begin{equation*}
        \norm{q_{\widetilde{m}_{n}}-q_{f}}_{\partial\Lambda_{i}(\epsilon)}\leq C_{4}\gamma_{n}.
    \end{equation*}
    Hence, if $\gamma_{n}\leq C^{-1}\epsilon$ for a sufficiently large constant $C$ independent of $n$ and $\epsilon$, then
    \begin{equation*}
        \abs{q_{\widetilde{m}_{n}}(z)-q_{f}(z)}<\abs{q_{f}(z)}\quad\text{for all}\quad  z\in\partial\Lambda_{i}(\epsilon).
    \end{equation*}
    Rouch\'e's theorem implies that $q_{\widetilde{m}_{n}}$ and $q_{f}$ have the same number of zeros in $\Lambda_{i}(\epsilon)$. Since $q_{f}$ has exactly $m_{i,1}$ zeros in $\Lambda_{i}(\epsilon)$, the function $r_{n}$ has exactly $m_{i,1}$ poles in $\Lambda_{i}(\epsilon)$, counted with multiplicity.
\end{proof}

To show the sharpness of the exponent, let us consider the following example: 
\begin{equation}
    \label{eq:example}
    f(z) = \frac{1}{(z-0.5)^{m}(z-2)(z-3)^{2}}.
\end{equation}
Sampling 100 points on $\partial\region\defi\{z\in\C \mid\abs{z}=1\}$ and using AAA to compute its rational approximation, we collect the accuracy of poles in $\region$ in \cref{tab:multi}, verifying our theoretical findings.
\begin{table}[htbp]
    \caption{Accuracy for computing a pole with multiplicity $m$.}
    \label{tab:multi}
    \centering
    \begin{tabular}{c|cccccccc}
        \toprule
        multiplicity $m$ &1 & 2& 3& 4 & 5 &6 & 7 & 8 \\ \midrule
        accuracy &3e-16 & 7e-9 & 8e-6 & 1e-4 & 5e-4 & 2e-3 & 6e-3 & 1e-2 \\ \midrule
        $\mpr^{1/m}$&2e-16 & 2e-8 & 6e-6 & 1e-4 & 7e-4 & 3e-3 & 6e-3 & 1e-2\\ \bottomrule
    \end{tabular}
\end{table}

In view of \cref{thm:convPolesIn}, there are mainly two issues that could prevent us from reaching satisfactory accuracy: a large pre-factor $C$ and high multiplicity $m_{i,1}$. To further improve the accuracy, as suggested in \cite{bruno2024evaluation}, one can zoom in on the domain.
Actually, by scaling arguments, we can show the benefit of approximation on a domain with diameter $\delta\ll 1$. 
More specifically, let $f^{(\delta)}(z)=f(\delta z)$, $r_{n}^{(\delta)}(z)=r_{n}(\delta z)$ and $\region^{(\delta)}=\{z/\delta\mid z\in\region\}$. Then we have that 
\begin{equation*}
    \norm{f^{(\delta)}(z)-r_{n}^{(\delta)}(z)}_{\partial\region^{(\delta)}} = \norm{f-r_{n}}_{\partial\region}\leq \gamma_{n}.
\end{equation*}
Since the diameter of $\region^{(\delta)}$ is $1$, using \cref{thm:convPolesIn}, we know that $r_{n}^{(\delta)}$ contains $m_{i,1}$ poles in 
\begin{equation*}
    \Lambda_{i}^{(\delta)}(\epsilon/\delta^{m_{i,1}})\defi \{z\in\region^{(\delta)}\mid \abs{z-\lambda_{i}/\delta}\leq \epsilon^{1/m_{i,1}}/\delta\},
\end{equation*}
where we use the fact that $\lambda_{i}/\delta $ is a pole of $f^{(\delta)}$ with multiplicity $m_{i,1}$.
Thus, $r_{n}$ contains $m_{i,1}$ poles in 
\begin{equation*}
    \{\delta z\mid z\in\Lambda_{i}^{(\delta)}(\epsilon/\delta^{m_{i,1}})\} = \{z\in\region\mid \abs{z-\lambda_{i}}\leq  \epsilon^{1/m_{i,1}}\}=\Lambda_{i}(\epsilon).
\end{equation*}
In summary, the scaling argument shows that the error in the original
coordinates is bounded by $\delta\bigl(C_{\delta}\gamma_{n}\bigr)^{1/m_{i,1}}$, where $C_{\delta}$ is the constant in \cref{thm:convPolesIn} applied to $f^{(\delta)}$ and $D^{(\delta)}$. Therefore, provided that $C_{\delta}$ remains bounded under zooming in, the accuracy of the approximate eigenvalues improves linearly with the diameter $\delta$ of the computational domain.
In particular, when $\gamma_{n}=O(\mpr)$ and $C_{\delta}$ is uniformly bounded, the accuracy improves from $(C\mpr)^{1/m_{i,1}}$ to
$\delta(C\mpr)^{1/m_{i,1}}$.
To illustrate with an example, we consider the same function defined in \cref{eq:example}. Using 
\begin{equation*}
    \partial\region^{(\delta)}\defi \{z\in\C\mid \abs{z-(1-\delta)/2}=\delta\},
\end{equation*}
where $\delta = 10^{-k}$ for $k=0,\dotsc,8$, numerical results in \cref{tab:zoomin} show the effectiveness of zooming in, suggesting that a small $\delta$ can significantly improve the accuracy.
Note that, under the uniform boundedness assumption above, the linear
scaling with $\delta$ is independent of the multiplicity of the pole.

\begin{table}[htbp]
    \caption{Accuracy for computing a pole with multiplicity $m=8$ in domain with radius $\delta$.}
    \label{tab:zoomin}
    \centering
    \begin{tabular}{c|ccccccccc}
        \toprule
        radius $\delta$ &1e-0 & 1e-1& 1e-2& 1e-3 & 1e-4 &1e-5 & 1e-6 & 1e-7&1e-8 \\ \midrule
        accuracy &1e-2 & 1e-3 & 5e-3 & 2e-5 & 5e-5 & 5e-6 & 5e-7 & 5e-8 &7e-10\\ \bottomrule
    \end{tabular}
\end{table}

\subsection{Fast convergence in domains with few poles}
The convergence result in \cref{thm:convPolesIn} does not fully explain the behavior of poles because of the constant $C$. Indeed, it depends strongly on the number of poles in $\region$. In this part, we further explore the convergence behavior with a focus on the situation where $f$ contains only simple poles within a  domain $\region$ of unit diameter. As shown above, under mild assumptions, the corresponding rational approximation $r_{n}$ also possesses the same number of poles in $\region$. For simplicity of notation, we drop the subscript $n$ of $r_{n}$ in this part.

Our analysis starts with a simple but illustrative situation, that is, $f$ only contains a single pole in $\region$. The following proposition establishes that the accuracy of the computed pole is proportional to the approximation quality on $\partial\region$.

\begin{proposition}
    \label{prop:singlepole}
    Let $\region\subset \C$ be a nonempty domain enclosed by Jordan curve. Consider the following two meromorphic functions:
    \begin{equation*}
        f(z) = \frac{w}{z-\lambda}+h(z)
        \quad\text{and}\quad 
        r(z) = \frac{\widehat{w}}{z-\widehat{\lambda}}+\widehat{h}(z),
    \end{equation*}
    where $\lambda$ and $\widehat{\lambda}$ are in $\region$ and $h$ and $\widehat{h}$ are holomorphic in a neighborhood of $\region$. Let $\varepsilon\defi \norm{f-r}_{\partial\region}$. Then 
\begin{equation*}
    \abs{\lambda-\widehat{\lambda}} 
    \leq (C/\abs{w})\varepsilon,
\end{equation*}
where $C$ depends only on the domain $\region$.
\end{proposition}
\begin{proof}
By the residue theorem, we have 
\begin{equation*}
    \frac{1}{2\pi\mi}\oint_{\partial\region} (z-\widehat{\lambda})f(z)-(z-\widehat{\lambda})r(z) \de z= (\lambda-\widehat{\lambda}) w. 
\end{equation*}
Plugging in the condition $\varepsilon=\norm{f-r}_{\partial\region}$, we finish the proof by the triangle inequality:
\begin{equation*}
    \abs{\lambda-\widehat{\lambda}} \leq \frac{\norm{f-r}_{\partial\region}}{2\pi\abs{w}}\oint_{\partial\region}\abs{z-\widehat{\lambda}}\abs{\de z}
    \leq (C/\abs{w})\varepsilon,
\end{equation*}
where $C$ depends only on the domain $\region$.
\end{proof}

A natural question arising from \cref{prop:singlepole} is whether a similar convergence result can be obtained when $f$ contains several poles within $\region$.
Furthermore, suppose we have $s$ rational functions $r_{i}$ approximating $f_{i}$ for $1 \leq i \leq s$, where the functions $\{f_{i}\}_{i=1}^{s}$ share the same poles, and likewise for $\{r_{i}\}_{i=1}^{s}$.
Can we achieve improved convergence as $s$ increases?
This situation naturally arises when applying the set-valued AAA algorithm to approximate a matrix-valued surrogate function $F$ of size $b \times b$, which provides $s = b^{2}$ scalar approximations.
The following theorem gives an affirmative answer to both of these questions.

\begin{theorem}
    \label{thm:singlepole}
    Let $\region\subset \C$ be a nonempty domain enclosed by Jordan curve. 
    Consider the following meromorphic functions (sharing their poles) for $1\leq i\leq s$:
    \begin{equation*}
        f_{i}(z) = \sum_{j=1}^{m}\frac{w_{ij}}{z-\lambda_{j}}+h_{i}(z)
        \quad\text{and}\quad 
        r_{i}(z) = \sum_{j=1}^{m}\frac{\widehat{w}_{ij}}{z-\widehat{\lambda}_{j}}+\widehat{h}_{i}(z),
    \end{equation*}
    where $\lambda_{j}$ and $\widehat{\lambda}_{j}$ are in $\region$ for $1\leq j\leq m$, and $h_{i}$ and $\widehat{h}_{i}$ are holomorphic in a neighborhood of $\region$. Let 
    \begin{equation*}
        \mathcal{F}(z) = [f_{1}(z),\dotsc,f_{s}(z)]^{\Ttran}
        \quad\text{and}\quad 
        \mathcal{R}(z) = [r_{1}(z),\dotsc,r_{s}(z)]^{\Ttran}.
    \end{equation*}
    Denote the node polynomials associated with poles $\lambda_{1},\dotsc,\lambda_{m}$ and $\widehat{\lambda}_{1},\dotsc,\widehat{\lambda}_{m}$ by 
    \begin{equation}
        \label{eq:thmpoledefp}
        p(t) = \prod_{j=1}^{m}(t-\lambda_{j}) = \sum_{j=0}^{m}p_{j}t^{j}
        \quad \text{and}\quad 
        \widehat{p}(t) = \prod_{j=1}^{m}(t-\widehat{\lambda}_{j}) = \sum_{j=0}^{m}\widehat{p}_{j}t^{j}.
    \end{equation}
    If $\varepsilon\defi \max\limits_{z\in\partial\region}\norm{\mathcal{F}(z)-\mathcal{R}(z)}$ is sufficiently small, then 
\begin{equation*}
    \norm{\mathbf{p}-\widehat{\mathbf{p}}}=\order(\varepsilon)
    \quad\text{where}\quad 
        \mathbf{p}=[p_{0},\dotsc,p_{m}]^{\Ttran}
    \quad\text{and}\quad 
        \widehat{\mathbf{p}}=[\widehat{p}_{0},\dotsc,\widehat{p}_{m}]^{\Ttran}.
\end{equation*}  
\end{theorem}

The convergence of $\widehat{\mathbf{p}}$ to $\mathbf{p}$ implies the convergence of poles, although their dependence may be arbitrarily ill-conditioned. This can be obtained by applying the Bauer--Fike theorem \cite{Bauer1960} to the companion matrices associated with monic polynomials $\widehat{p}$ and $p$.

\begin{proof}[Proof of \cref{thm:singlepole}]
    Consider the $k$-th order moments for $0\leq k\leq 2m-1$:
    \begin{equation*}
        M_{k}\defi\frac{1}{2\pi\mi}\oint_{\partial\region}z^{k}\mathcal{F}(z)\de z = \frac{1}{2\pi\mi}
        \oint_{\partial\region}z^{k}\begin{bmatrix}
            f_{1}(z)\\ 
            \vdots\\ 
            f_{s}(z)
        \end{bmatrix}\de z = \sum_{j=1}^{m}\begin{bmatrix}
            w_{1j}\\ 
            \vdots\\ 
            w_{sj}
        \end{bmatrix}\lambda_{j}^{k}
        =W\begin{bmatrix}
            \lambda_{1}^{k}\\ 
            \vdots\\ 
            \lambda_{m}^{k}
        \end{bmatrix}\in\C^{s},
    \end{equation*}
    where $W\in\C^{s\times m}$ consists of $w_{ij}$.
    Defining $\widehat{M}_{k}$ as the moment related to $r_{i}$ similarly, the triangle inequality gives $\norm{M_{k}-\widehat{M}_{k}}=\order(\varepsilon)$:
    \begin{equation*}
        \norm{M_{k}-\widehat{M}_{k}} = \frac{1}{2\pi}\Bignorm{\oint_{\partial\region}z^{k}\Bigl(\mathcal{F}(z)-\mathcal{R}(z)\Bigr)\de z} \leq \frac{\norm{\mathcal{F}(z)-\mathcal{R}(z)}_{\partial\region}}{2\pi}\oint_{\partial\region}\abs{z}^{k}\abs{\de z}=\order(\varepsilon).
    \end{equation*}
    Let $\van$ and $\widehat{\van}$ be Vandermonde matrices related to the poles:
    \begin{equation*}
        \van \defi \begin{bmatrix}
            1 & \lambda_{1} &\cdots &\lambda_{1}^{m}\\ 
            1 & \lambda_{2} &\cdots &\lambda_{2}^{m}\\ 
            \vdots &\vdots &\ddots &\vdots \\ 
            1 & \lambda_{m} &\cdots &\lambda_{m}^{m} 
        \end{bmatrix}
        \quad\text{and}\quad 
        \widehat{\van} \defi \begin{bmatrix}
            1 & \widehat{\lambda}_{1} &\cdots &\widehat{\lambda}_{1}^{m}\\ 
            1 & \widehat{\lambda}_{2} &\cdots &\widehat{\lambda}_{2}^{m}\\ 
            \vdots &\vdots &\ddots &\vdots \\ 
            1 & \widehat{\lambda}_{m} &\cdots &\widehat{\lambda}_{m}^{m} 
        \end{bmatrix}\in\C^{m\times (m+1)}.
    \end{equation*}
    Define $\Lambda=\diagM{\lambda_{1},\dotsc,\lambda_{m}}$ and $\widehat{\Lambda}=\diagM{\widehat{\lambda}_{1},\dotsc,\widehat{\lambda}_{m}}$.
    Assembling the moments we obtain
    \begin{equation*}
        \begin{bmatrix}
            W\\ 
            W\Lambda\\ 
            \vdots\\ 
            W\Lambda^{m-1} 
        \end{bmatrix}\van = \begin{bmatrix}
            M_{0}&M_{1}&\cdots &M_{m}\\ 
            M_{1}&M_{2}&\cdots &M_{m+1}\\ 
            \vdots &\vdots &\ddots &\vdots\\ 
            M_{m-1}&M_{m}&\cdots & M_{2m-1}
        \end{bmatrix} =: \han\in\C^{ms\times (m+1)}. 
    \end{equation*}
    Now let us consider the node polynomials associated with poles in $\region$. By the definition of $p$ in \cref{eq:thmpoledefp}, we have 
    \begin{equation*}
        \han\cdot\mathbf{p} = \begin{bmatrix}
            W\\ 
            W\Lambda\\ 
            \vdots\\ 
            W\Lambda^{m-1} 
        \end{bmatrix}\van \cdot\mathbf{p} = \begin{bmatrix}
            W\\ 
            W\Lambda\\ 
            \vdots\\ 
            W\Lambda^{m-1} 
        \end{bmatrix}\begin{bmatrix}
            p(\lambda_{1})\\ 
            \vdots\\ 
            p(\lambda_{m})
        \end{bmatrix}   =0,
    \end{equation*}
    and similarly $\widehat{\han}\widehat{\mathbf{p}} = 0$, 
    where $\widehat{\han}$ is defined similarly to $\han$ with $f_{i}$ replaced by $r_{i}$. Recalling that $\norm{M_{k}-\widehat{M}_{k}}=\order(\varepsilon)$ and $\han$ and $\widehat{\han}$ are formed by stacking $M_{k}$ and $\widehat{M}_{k}$, respectively, we see that $\norm{\han-\widehat{\han}}=\order(\varepsilon)$.
    For a sufficiently small $\varepsilon$, we know that the second smallest singular value of $\han$ satisfies $\sigma_{m}\geq \norm{\han-\widehat{\han}}$. In turn, Weyl's inequality yields that  the second smallest singular value of $\widehat{\han}$ is also positive. Thus, $\mathbf{p}$ and $\widehat{\mathbf{p}}$ are (right) singular vectors corresponding to the smallest singular values (which are zero) of $\han$ and $\widehat{\han}$, respectively. Then the theorem is proved by the perturbation result of singular vectors by Wedin \cite{Wedin1972}.
\end{proof}

The condition $\varepsilon$ being sufficiently small (such that $\sigma_{m}\geq \norm{\han-\widehat{\han}}=\order(\varepsilon)$) is a very strong assumption when $m$ is (moderately) large, since the smallest singular value of a Vandermonde matrix typically behaves as $\order(\gap^{m-1})$ \cite{Gautschi1962,beckermann2000condition}, where $\gap$ denotes the minimal relative gap among $\lambda_{1}, \dotsc, \lambda_{m}$.
Using a single probing vector makes the situation even worse, as a (nearly) square Hankel matrix effectively involves two Vandermonde matrices:
\begin{equation*}
    \begin{aligned}
        \han &= \begin{bmatrix}
        w_{11} &w_{12}& \cdots & w_{1m}\\ 
        w_{11}\lambda_{1} &w_{12}\lambda_{2}& \cdots & w_{1m}\lambda_{m}\\
        \vdots & \vdots & \ddots &\vdots \\ 
        w_{11}\lambda_{1}^{m-1}& w_{12}\lambda_{2}^{m-1} &\cdots & w_{1m}\lambda_{m}^{m-1}
    \end{bmatrix}\van \\
    &= \Bigl(\van\cdot \begin{bmatrix}
        I_{m}\\ 0
    \end{bmatrix}\Bigr)^{\Ttran}\diagM{w_{11},\dotsc,w_{1m}}\van.
    \end{aligned}
\end{equation*}
One possible way to weaken this assumption is by representing the polynomials $p$ and $\widehat{p}$ using alternative polynomials, such as Chebyshev polynomials.

Besides the probabilistic analysis in \cref{thm:prob}, \cref{thm:singlepole} provides an alternative explanation for the difference between sketching with a single vector and with multiple vectors: block probing improves the first Vandermonde matrix. In fact, this improvement is closely related to the so-called cluster robustness of the (randomized) small-block Lanczos method, whose analysis is intricate even when all $\lambda_{j}$ are real \cite{shao2025structural,chen2025does}. Empirical results indicate that a small block size, say $4$, is sufficient to significantly enhance the cluster robustness.

The key insight from \cref{thm:singlepole} is that when $m$, the number of poles within $\region$, is large, the accuracy of the approximate poles is expected to deteriorate, unless the poles of $f$ follow certain special distributions, such as being uniformly placed on a circle. 
In contrast, the zooming-in technique mitigates this by partitioning $\region$, thereby reducing $m$ in each subdomain and significantly enhancing the accuracy of the computed results.

As a numerical example, we compute all $61$ eigenvalues in $\region\defi \{a+b\mi\mid 0.2\leq a \leq 1.2\text{ and } 0.1\leq b\leq 1.1\}$ of the nonlinear problem \texttt{butterfly} from NLEVP \cite{Betcke2013}, which is a $64\times 64$ quartic eigenvalue problem.
Suppose the domain $\region$ is divided into $N^{2}$ small subdomains, as shown in \cref{fig:butterflydm}, where each subdomain is a square of size $1/N \times 1/N$. We sample 200 uniformly distributed points along the boundary of each subdomain and compute the eigenvalues within each subdomain using the poles of the rational approximation obtained from the (set-valued) AAA algorithm.
The results obtained from \texttt{polyeig}, which is based on the QZ method for a linearized generalized eigenvalue problem, are used as the reference solution. The error is measured as the maximum distance between the computed eigenvalues and the reference solution.

\begin{figure}[htbp]
    \centering
    \includegraphics[width=\figsizeD]{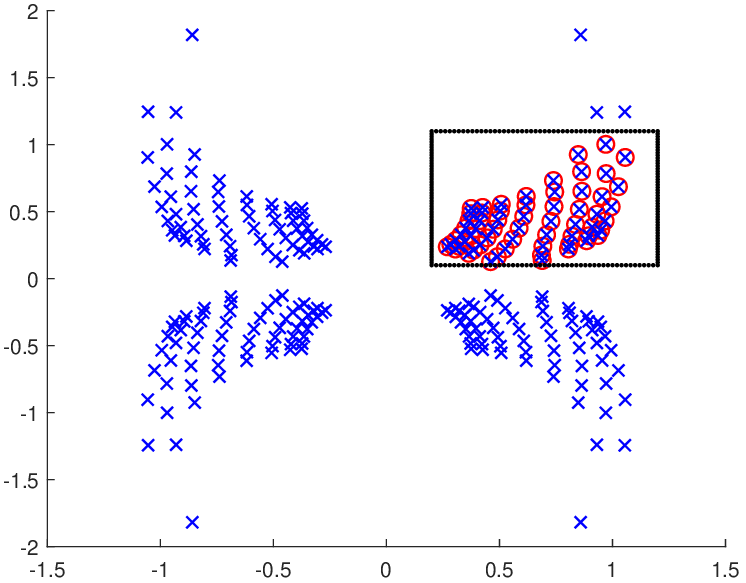}
    \includegraphics[width=\figsizeD]{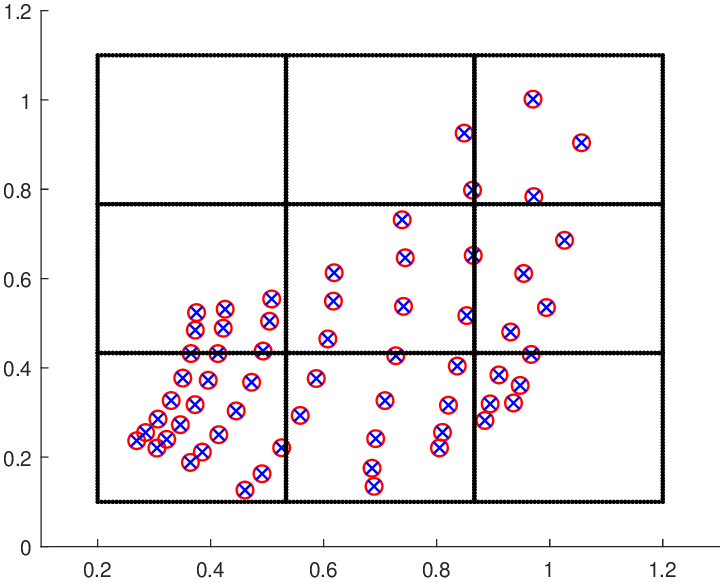}
    \caption{Eigenvalues of butterfly problem (left) and subdomains for $N=3$ (right).}
    \label{fig:butterflydm}    
\end{figure}

Varying the block size of probing matrices over $1\leq b\leq 3$ and the parameter for dividing domains over $1\leq N\leq 5$, we collect numerical results in \cref{fig:butterfly}. We see that zooming in can substantially improve the accuracy, as predicted by the theory.
Rational approximations remain accurate when only a few well-separated poles lie within the domain.
By zooming in, pole clusters are separated, effectively reducing the number of poles per domain and improving the overall accuracy.

\begin{figure}[htbp]
    \centering
    \includegraphics[width=\textwidth]{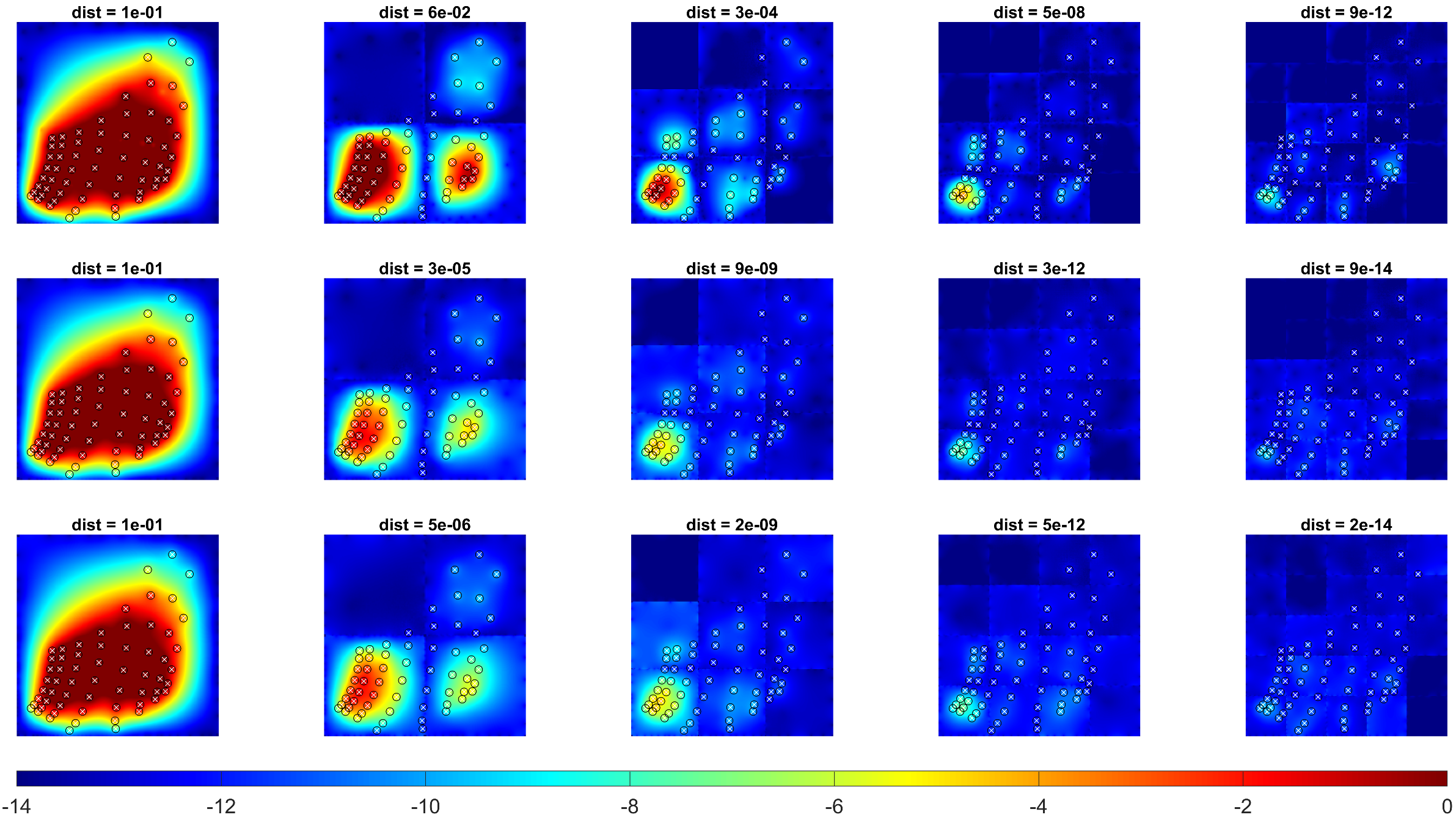}
    \caption{\texttt{Butterfly} eigenvalue problems. The $(b, N)$ subfigure represents the computational result obtained using block probing of size $b$ with sampling over subdomains of size $1/N$. The heat map shows the logarithmic error between $F$ and its rational approximation, while ``dist'' denotes the maximum error in the computed eigenvalues.} 
    \label{fig:butterfly}    
\end{figure}

\section{Stability of polefinding via a generalized eigenproblem}
\label{sec:stability}
Suppose that we have a good rational approximation $R$ in the barycentric form \cref{eq:defR}. This section discusses the stability of polefinding.
Even though the function $R$ defined in \cref{eq:defR} is matrix-valued, its denominator is scalar. Since the support points $z_{k}$ are removable singularities of $R$, and assuming that the scalar denominator and the entries of the matrix-valued numerator have no nonconstant common factor, the poles of $R$ are precisely the zeros of its denominator, which is a scalar function:
\begin{equation}
    \label{eq:defbc}
    q(z) = \sum_{k=1}^{K}\frac{w_{k}}{z-z_{k}},\quad\text{where}\quad \sum_{k=1}^{K}\abs{w_{k}}^{2}=1.
\end{equation}
In the following, we look into the stability of rootfinding for $q$.
\begin{remark}
    The stability results in this section are concerned with the more general problem of polefinding from a rational function in barycentric form \cref{eq:defR}, or equivalently, the rootfinding of $q$ in \cref{eq:defbc}. They are not necessarily related to the (set-valued) AAA method.
\end{remark}
\subsection{Backward stability of rootfinding} We first define the backward stability of rootfinding from a barycentric form.
\begin{definition}
    Given the barycentric form \cref{eq:defbc} of a rational function $q$, we say a computed root $\widehat{\lambda}$ is (backward) stable if there exist $\widehat{w}_{k}$ for $1\leq k\leq K$ satisfying $\abs{w_{k}-\widehat{w}_{k}}=\order(\mpr)$, such that 
    \begin{equation*}
        \widehat{q}(\widehat{\lambda})\defi \sum_{k=1}^{K}\frac{\widehat{w}_{k}}{\widehat{\lambda}-z_{k}} = 0.
    \end{equation*}
\end{definition}

As discussed in \cite[Sec.~2.3.3]{klein2012}, roots from a barycentric form \cref{eq:defbc} can be recovered by the finite generalized eigenvalues of the following matrix pencil:
\begin{equation}
    \label{eq:defgevp}
    \begin{bmatrix}
        0 & w_{1} & w_{2} & \cdots & w_{K}\\ 
        1 & z_{1}&&& \\ 
        1& & z_{2}&&\\ 
        \vdots &&&\ddots &\\ 
        1 &&&&z_{K}
    \end{bmatrix}-\lambda \begin{bmatrix}
        0&&&&\\
        &1&&&\\
        &&1&&\\
        &&&\ddots&\\
        &&&&1
    \end{bmatrix}
\end{equation}
The following theorem shows that, if the generalized eigenvalues of \cref{eq:defgevp} are solved in a backward stable manner such as by the QZ algorithm~\cite{moler1973algorithm}, that is, the computed eigenvalues of a matrix pencil $A-\lambda B$ are the exact eigenvalues of $(A+\Delta A)-\lambda(B+\Delta B)$ for small $\Delta A,\Delta B$, under mild assumptions, the computed roots are backward stable.

\begin{theorem}
    \label{thm:bs}
    Given a barycentric form of $q$ as in \cref{eq:defbc} with $\abs{z_{k}}=\order(1)$ for all $1\leq k\leq K$. Let $\widehat{\lambda}$ be a computed generalized eigenvalue of \cref{eq:defgevp} from a backward stable method. 
    If there exists an integer $\widehat{k}$ such that $1\leq \widehat{k}\leq K$ and 
    \begin{equation}
        \label{eq:aspbs}
        \min_{1\leq k\neq \widehat{k}\leq K} \abs{\widehat{\lambda}-z_{k}}=\Theta(1),
    \end{equation}
    then $\widehat{\lambda}$ is a backward stable root of $q$.
\end{theorem}

The condition $\abs{z_{k}}=\order(1)$ is not essential since we can always shift or rescale the support points before computing the rational approximation. 
The aim of this condition is to obtain an $\order(\mpr)$ forward error of the computed eigenvectors from the backward stability of eigensolvers.
By contrast, the condition $\abs{\widehat{\lambda}-z_{k}}=\Theta(1)$ is indeed a genuine assumption for the backward stability. 
It permits one support point $z_{\widehat{k}}$ to be close to the computed eigenvalue $\widehat{\lambda}$, while requiring all other support points to be sufficiently distant from it. 
The unfavorable situation, where $\widehat{\lambda}$ lies close to two support points, can be avoided by choosing well-separated support points $\{z_{k}\}$.

\begin{proof}[Proof of \cref{thm:bs}]
    Let $\widehat{x}=[\widehat{x}_{0},\widehat{x}_{1},\dotsc,\widehat{x}_{K}]^{\Ttran}$ be the corresponding computed (unit-norm) eigenvector. By the backward stability of the eigensolver, we have 
    \begin{equation}
        \label{eq:pfbs}
        \mpr_{0}\defi \sum_{k=1}^{K}w_{k}\widehat{x}_{k}=\order(\mpr)\quad \text{and}\quad  \mpr_{k}\defi \widehat{x}_{0}+z_{k}\widehat{x}_{k}-\widehat{\lambda}\widehat{x}_{k}=\order(\mpr)\quad \forall\,1\leq k\leq K.
    \end{equation}
    Thus, for $1\leq k\neq \widehat{k}\leq K$, the second relation in \cref{eq:pfbs} yields 
    \begin{equation}
        \label{eq:pfbs1}
        \widehat{x}_{k} = \frac{\widehat{x}_{0}-\mpr_{k}}{\widehat{\lambda}-z_{k}} 
        = \frac{\widehat{x}_{0}-\mpr_{\widehat{k}}}{\widehat{\lambda}-z_{k}}+\frac{\mpr_{\widehat{k}}-\mpr_{k}}{\widehat{\lambda}-z_{k}} 
        = \frac{\widehat{\lambda}-z_{\widehat{k}}}{\widehat{\lambda}-z_{k}}\cdot \widehat{x}_{\widehat{k}}+\frac{\mpr_{\widehat{k}}-\mpr_{k}}{\widehat{\lambda}-z_{k}}.
    \end{equation}
    By the assumption \cref{eq:aspbs} and $\abs{z_{k}}=\order(1)$, we know that 
    \begin{equation*}
        \frac{\widehat{\lambda}-z_{\widehat{k}}}{\widehat{\lambda}-z_{k}} = \order(1)\quad\text{and}\quad 
        \frac{\mpr_{\widehat{k}}-\mpr_{k}}{\widehat{\lambda}-z_{k}}=\order(\mpr).
    \end{equation*}
    In turn, the normalization condition $\norm{\widehat{x}}=1$ implies that $\widehat{x}_{\widehat{k}}=\Theta(1)$.
    Plugging \cref{eq:pfbs1} into the first relation in \cref{eq:pfbs}, we obtain 
    \begin{equation*}
        \Bigl(w_{\widehat{k}}+(\widehat{\lambda}-z_{\widehat{k}})\sum_{k\neq \widehat{k}}\frac{w_{k}}{\widehat{\lambda}-z_{k}} \Bigr)\widehat{x}_{\widehat{k}}
        = \widehat{\mpr}\defi \mpr_{0}-\sum_{k\neq \widehat{k}}w_{k}\cdot \frac{\mpr_{\widehat{k}}-\mpr_{k}}{\widehat{\lambda}-z_{k}}=\order(\mpr).
    \end{equation*}
    Define 
    \begin{equation*}
        \widehat{w}_{\widehat{k}}\defi w_{\widehat{k}}-\widehat{\mpr}/\widehat{x}_{\widehat{k}}=w_{\widehat{k}}+\order(\mpr)
        \quad\text{and}\quad  
        \widehat{w}_{k}=w_{k}\quad\text{for}\quad 1\leq k\neq\widehat{k}\leq K.
    \end{equation*}
    We finish the proof as follows:
    \begin{equation*}
        \begin{aligned}
            \sum_{k=1}^{K}\frac{\widehat{w}_{k}}{\widehat{\lambda}-z_{k}} &= \frac{1}{\widehat{\lambda}-z_{\widehat{k}}} \Bigl(\widehat{w}_{\widehat{k}}+(\widehat{\lambda}-z_{\widehat{k}})\sum_{k\neq\widehat{k}}\frac{\widehat{w}_{k}}{\widehat{\lambda}-z_{k}}\Bigr)\\ 
        &= \frac{1}{\widehat{\lambda}-z_{\widehat{k}}} \Bigl(w_{\widehat{k}}+(\widehat{\lambda}-z_{\widehat{k}})\sum_{k\neq\widehat{k}}\frac{w_{k}}{\widehat{\lambda}-z_{k}}-\widehat{\mpr}/\widehat{x}_{\widehat{k}}\Bigr)=0.       
        \end{aligned}
    \end{equation*}

\end{proof}

Now we provide a numerical example to show that, without the condition \cref{eq:aspbs}, the computed eigenvalue $\widehat{\lambda}$ can be backward unstable. Let 
\begin{equation*}
    \widehat{z} = \Bigl[\frac{1}{\widehat{\lambda}-z_{1}},\dotsc,\frac{1}{\widehat{\lambda}-z_{K}}\Bigr]^{\Ttran}\quad\text{and}\quad w = [w_{1},\dotsc,w_{K}]^{\Htran}.
\end{equation*} 
Then the backward stability is equivalent to the existence of a perturbation $\delta w\in\C^{K}$ with $\norm{\delta w}=\order(\mpr)$ such that $(w+\delta w)^{\Htran}\widehat{z}=0$. We can actually write down the minimal perturbation $\delta w$ and its norm as 
\begin{equation*}
    \delta w = -\frac{\widehat{z}\widehat{z}^{\Htran}w}{\norm{\widehat{z}}^{2}}\quad\text{and}\quad \norm{\delta w} = \frac{\abs{\widehat{z}^{\Htran}w}}{\norm{\widehat{z}}},
\end{equation*}
which can be computed numerically once we know $\widehat{\lambda}$, the computed eigenvalue. Consider the following example: $z_{1}=1$, $z_{2}=1+\alpha^{-1}$,
\begin{equation*}
    z_{k}=\frac{k}{\alpha}\quad\text{for}\quad  k=3,\dotsc,K\quad\text{and}\quad w =\frac{[1,10^{-1},10^{-12},\dotsc,10^{-12}]^{\Ttran}}{\norm{[1,10^{-1},10^{-12},\dotsc,10^{-12}]}}\in\C^{K}.
\end{equation*}
We set $K=100$, vary $\alpha$ over $\{10^{3},10^{4},\dotsc,10^{9}\}$, and solve the generalized eigenvalue problem by \texttt{eig} in Matlab. Note that for large $\alpha$, there exists an eigenvalue $\widehat{\lambda}$ close to both $z_{1}$ and $z_{2}$ since $\abs{z_{1}-z_{2}}=1/\alpha$. So we look into the behavior of this computed eigenvalue. The minimal perturbation $\norm{\delta w}$ and second smallest distance to support points $\min_{k\neq \widehat{k}}\abs{z_{k}-\widehat{\lambda}}$ are collected in  \cref{tab:bs}, showing that when there are two close support points, this procedure is not backward stable, and the backward error is inversely proportional to the gap $\min_{k\neq \widehat{k}}\abs{z_{k}-\widehat{\lambda}}$.

\begin{table}[htbp]
    \caption{Backward error of rootfinding.}
    \label{tab:bs}
    \centering
    \begin{tabular}{l|cccccccc}
        \toprule
        $\alpha$ &$10^{3}$& $10^{4}$ & $10^{5}$ &$10^{6}$ & $10^{7}$& $10^{8}$& $10^{9}$ \\ \midrule
        $\min\norm{\delta w}$ &2e-13 & 6e-12 & 1e-11 & 5e-10 & 6e-9 & 1e-8 & 2e-7\\ \midrule
        $\min_{k\neq \widehat{k}}\abs{z_{k}-\widehat{\lambda}}$&9e-4 & 9e-5 & 9e-6 & 9e-7 & 9e-8 & 9e-9 & 9e-10
\\ \bottomrule
    \end{tabular}
\end{table}

\subsection{Forward stability of rootfinding}
The backward stability result shows that the computed roots are roots of another rational function $\widehat{q}$, whose weights in the barycentric form are close to the original function $q$. The following forward stability result illustrates how perturbations in the weights influence the accuracy of their roots.
\begin{theorem}
    \label{thm:fs}
    Let $q(z)$ and $\widehat{q}(z)$ be two rational functions in barycentric forms:
    \begin{equation*}
        q(z) = \sum_{k=1}^{K}\frac{w_{k}}{z-z_{k}}
        \quad\text{and}\quad 
        \widehat{q}(z) = \sum_{k=1}^{K}\frac{\widehat{w}_{k}}{z-z_{k}},
    \end{equation*}
    where $\abs{\widehat{w}_{k}-w_{k}}=\order(\mpr)$ for all $1\leq k\leq K$ and $\sum_{k=1}^{K}\abs{w_{k}}^{2}=1$. Let $\lambda_{i}$ be a root of $q$ with multiplicity $m_{i,1}$. Then there exist $m_{i,1}$ roots $\widehat{\lambda}_{i,j}$ of $\widehat{q}$, such that 
    \begin{equation*}
        \widehat{\lambda}_{i,j}=\lambda_{i}+\order(\kappa_{i}\mpr^{1/m_{i,1}}), 
    \end{equation*} 
    where $ \kappa_{i}$ is defined via 
    \begin{equation*}
        \kappa_{i}^{-m_{i,1}} \defi \min_{1\leq k\leq K}\abs{\lambda_{i}-z_{k}}\cdot\Bigabs{\sum_{k=1}^{K}\frac{w_{k}}{(\lambda_{i}-z_{k})^{m_{i,1}+1}}} = \min_{1\leq k\leq K}\abs{\lambda_{i}-z_{k}}\cdot
        \Bigabs{\frac{q^{(m_{i,1})}(\lambda_{i})}{m_{i,1}!}}
        .
    \end{equation*}
\end{theorem}
\begin{proof} 
    When $\lambda_{i}$ is a root of multiplicity $m_{i,1}$, for $1\leq j\leq m_{i,1}$, it holds that 
    \begin{equation*}
        \sum_{k=1}^{K}\frac{w_{k}}{(\lambda_{i}-z_{k})^{j}} =(-1)^{j-1}\cdot \frac{q^{(j-1)}(\lambda_{i})}{(j-1)!}= 0.
    \end{equation*}
    Recalling that roots of $q$ are generalized eigenvalues of \cref{eq:defgevp}, we have 
        \begin{small}
    \begin{equation*}
        \begin{aligned}
        &\begin{bmatrix}
        0 & w_{1} & w_{2} & \cdots & w_{K}\\ 
        1 & z_{1}&&& \\ 
        1& & z_{2}&&\\ 
        \vdots &&&\ddots &\\ 
        1 &&&&z_{K}
    \end{bmatrix}
    \begin{bmatrix}
        1 &0&\cdots &0\\ 
        (\lambda_{i}-z_{1})^{-1}&(\lambda_{i}-z_{1})^{-2}&\cdots &(\lambda_{i}-z_{1})^{-m_{i,1}}\\
        \vdots&\vdots&\ddots &\vdots \\ 
        (\lambda_{i}-z_{K})^{-1}&(\lambda_{i}-z_{K})^{-2}&\cdots &(\lambda_{i}-z_{K})^{-m_{i,1}}
    \end{bmatrix}
    \\ =&
    \begin{bmatrix}
        0&\\ 
        & I
    \end{bmatrix}
    \begin{bmatrix}
        1 &0&\cdots &0\\ 
        (\lambda_{i}-z_{1})^{-1}&(\lambda_{i}-z_{1})^{-2}&\cdots &(\lambda_{i}-z_{1})^{-m_{i,1}}\\
        \vdots&\vdots&\ddots &\vdots \\ 
        (\lambda_{i}-z_{K})^{-1}&(\lambda_{i}-z_{K})^{-2}&\cdots &(\lambda_{i}-z_{K})^{-m_{i,1}}
    \end{bmatrix}
    \begin{bmatrix}
        \lambda_{i} & -1 &&&\\ 
        &\lambda_{i} & -1 &&\\ 
        &&\ddots &\ddots &\\ 
        &&&\lambda_{i} & -1\\ 
        &&&&\lambda_{i}
    \end{bmatrix},
        \end{aligned}
    \end{equation*}
        \end{small}
    implying that $\lambda_{i}$ has geometric multiplicity $1$. Denote the left and right (normalized) eigenvectors corresponding to $\lambda_{i}$ by 
    \begin{equation*}
        \begin{aligned}
        x_{L}^{\Htran} &= \kappa_{i}^{m_{i,1}}\cdot \min_{1\leq k\leq K}\abs{\lambda_{i}-z_{k}}\cdot \bigl[1,\frac{w_{1}}{\lambda_{i}-z_{1}},\dotsc,\frac{w_{K}}{\lambda_{i}-z_{K}}\bigr],\\ 
        x_{R} &= \bigl[1,\frac{1}{\lambda_{i}-z_{1}},\dotsc,\frac{1}{\lambda_{i}-z_{K}}\bigr]^{\Ttran},
        \end{aligned}
    \end{equation*}
    where the normalization comes from imposing that 
    \begin{equation*}
        \left|{x_{L}^{\Htran}\begin{bmatrix}
            0&&&\\ 
            &1&&\\ 
            &&&\ddots&\\ 
            &&&&1
        \end{bmatrix}\begin{bmatrix}
            0\\ (\lambda_{i}-z_{1})^{-m_{i,1}}\\ 
            \vdots\\ (\lambda_{i}-z_{K})^{-m_{i,1}}
        \end{bmatrix}}\right|=1. 
    \end{equation*}
    Then we can use the first order perturbation theory for defective (generalized) eigenvalues in \cite{Kressner2009,DeTeran2008} to know that there are $m_{i,1}$ (generalized) eigenvalues $
        \widehat{\lambda}_{i}$ of \cref{eq:defgevp} with $w_{k}$ replaced by $\widehat{w}_{k}$, such that  
    \begin{equation*}
        \widehat{\lambda}_{i,j}=\lambda_{i} + \order\bigl((\abs{e_{1}^{\Ttran}x_{L}}\norm{x_{R}}\mpr)^{1/m_{i,1}}\bigr) =\lambda_{i} + \order\bigl(\kappa_{i}\mpr^{1/m_{i,1}}\bigr).
    \end{equation*}
    Here we use the fact that the perturbation on the matrix pencil \cref{eq:defgevp} only happens in the first row.
    The proof is completed by noticing that $\widehat{\lambda}_{i}$ are also roots of $\widehat{q}$. 
\end{proof}

\cref{thm:fs} suggests that, for \cref{algo}, the best forward accuracy we can expect is $\order(\mpr^{1/m_{i,1}})$, where $m_{i,1}$ is the maximum partial algebraic multiplicity corresponding to the eigenvalue $\lambda_{i}$. For example, consider the following function:
\begin{equation*}
    q_{m}(z) \defi \frac{2^{m+1}z^{2^{m}}}{z^{2^{m+1}}-1} = \sum_{k=1}^{2^{m+1}} \frac{(\omega_{m}^{k})^{2^{m}+1}}{z-\omega_{m}^{k}},\quad\text{where}\quad \omega_{m} = \exp(2^{-m}\pi\mi),
\end{equation*}
containing a root $0$ with multiplicity $2^{m}$. Taking $m=0,1,2,3$, we collect numerical results in \cref{tab:fs}, verifying the forward stability in \cref{thm:fs}. 
This result is in line with the theoretical accuracy limits of classical dense eigensolvers, such as QR and QZ, which are also known to suffer from reduced accuracy in the presence of defective eigenvalues.
\begin{table}[htbp]
    \caption{Forward error of rootfinding.}
    \label{tab:fs}
    \centering
    \begin{tabular}{l|cccc}
        \toprule
        $m$ &$0$& $1$ & $2$ &$3$ \\ \midrule
        $\min_{i}\abs{\widehat{\lambda}_{i}}$ &2e-16 & 8e-9 & 1e-4 & 1e-2\\ \midrule
        $\mpr^{2^{-m}}$&2e-16 & 1e-8 & 1e-4 & 1e-2
\\ \bottomrule
    \end{tabular}
\end{table}

\section{Concluding remarks}
In this paper, we analyzed the convergence of a nonlinear eigensolver based on rational approximations of the resolvent. In particular, we justified two techniques that improve accuracy: block probing and zooming in.

Several implementation-related questions remain for future work. These include developing robust methods to detect (partial) algebraic multiplicities from \cref{thm:sval}, and designing effective sampling strategies informed by \cref{thm:singlepole}. In addition, \cref{thm:bs}, together with the numerical experiments, indicates that polefinding in the AAA algorithm can be backward unstable when the sample points are poorly distributed. It would therefore be of interest to develop effective preprocessing strategies to mitigate this issue.

\section*{Acknowledgments}
The authors thank Nick Trefethen for helpful discussions on this work.
Shao thanks the Mathematical Institute at the University of Oxford for hosting him during a research visit in 2025, when this work was undertaken.
\bibliographystyle{siamplain}

\end{document}